\documentclass[11pt,A4paper]{article}

\usepackage[left=32mm,right=32mm,top=30mm,bottom=32mm]{geometry}
\usepackage{epsfig}
\usepackage{epstopdf}

\usepackage{xfrac}

\usepackage{color}
\usepackage{amsthm,amsmath,amssymb}
\usepackage{mathrsfs}
\usepackage{sectsty}
\usepackage[
	pdftitle={PDFTitle},
	pdfauthor={Hugo Parlier},
	ocgcolorlinks,
	linkcolor=linkred,
	citecolor=linkred,
	urlcolor=linkblue]
{hyperref}
  \ifx \LabelFigloaded\MYundefined
  \else
    \immediate\write16{ !!! LabelFig.tex ALREADY loaded.}
   \fi

  \def\LabelFigloaded{\relax}

  \chardef\LabelFigCatAt\the\catcode`\@
  \catcode`\@=11

 \let\LabelFigwlog@ld\wlog
 \def\wlog#1{\relax}

 \ifx\\\MYundefined@
    \let\\\relax
 \fi

 \def\N@wif{\csname newif\endcsname }
 \def\Temp@ {\N@wif\ifIN@}
 \ifx\INN@\MYundefined@
    \else \let\Temp@\relax
 \fi
 \Temp@

  \def\IN@{\expandafter\INN@\expandafter}
  \long\def\INN@0#1@#2@{\long\def\NI@##1#1##2##3\ENDNI@
    {\ifx\m@rker##2\IN@false\else\IN@true\fi}%
     \expandafter\NI@#2@@#1\m@rker\ENDNI@}
  \def\m@rker{\m@@rker}
 \def\Shifted@@#1#2#3{\setbox0=\hbox{#3}%
   \raise -\dp0\vbox {\kern-#2%
       \hbox {\kern#1\box0\kern-#1}%
           \kern#2}}

 \newcount\gridcount
 \newbox\auxGridbox@ \newbox\hGridbox@ \newbox\vGridbox@
 \newbox\Labelbox@ \newbox\auxLabelbox@
 \newbox\Coordinatebox@
 \newtoks\Labeltoks@
 \newdimen\Wdd@ \newdimen\Htt@

 \def\hRule@{\advance\gridcount -2%
   \vskip-.2pt\hrule\vskip-.2pt\vfil
   \llap{\smash{\raise -2.5pt
     \hbox{.\number\gridcount\kern2pt}}}%
           \vskip-.2pt\hrule\vskip-.2pt\vfil}

 \def\vRule@{\advance\gridcount 2%
   \hskip-.2pt\vrule\hskip-.2pt\hfil
   \setbox\auxGridbox@=\vbox to 0pt
      {\vskip \Htt@\vskip 2pt
        \hbox{\kern-3.5pt.\number\gridcount}\vss}%
      \wd\auxGridbox@=0pt \box\auxGridbox@
   \hskip-.2pt\vrule\hskip-.2pt\hfil}

 \def\PlaceGrid@@{\gridcount=10%
  \setbox\hGridbox@=%
    \hbox{\hbox{\GridSpider@{\hskip-.4pt\vrule}%
             \vbox to \Htt@{\offinterlineskip\parindent=\z@%
                \GridSpider@{\vskip-.4pt\hrule}\vfil
                \hRule@\hRule@\hRule@\hRule@
                  \vskip-.2pt\hrule\vskip-.2pt\vfil
                \hbox to \Wdd@{\hfil}%
             \GridSpider@{\hrule\vskip-.4pt}}%
         \GridSpider@{\vrule\hskip-.4pt}}}%
  \gridcount=0%
  \setbox\vGridbox@=
     \hbox{\vbox{\offinterlineskip\parindent=0pt\hsize=0pt
       \GridSpider@{\vskip-.4pt\hrule}%
             \hbox to \Wdd@{%
                \GridSpider@{\hskip-.4pt\vrule}\hfil
                \vtop to \Htt@{\vfil}%
                     \vRule@\vRule@\vRule@\vRule@
                     \hskip-.2pt\vrule\hskip-.2pt\hfil
             \GridSpider@{\vrule\hskip-.4pt}}%
         \GridSpider@{\hrule\vskip-.4pt}}}%
  \wd\hGridbox@=0pt\ht\hGridbox@=0pt
  \wd\vGridbox@=0pt\ht\vGridbox@=0pt
 \hbox{\box\hGridbox@\box\vGridbox@}%
  }

 \def\SetLabels#1\endSetLabels{%
   \Labeltoks@={#1}}

 \def\GridSpider@#1{#1}
 \let\PlaceGrid@\relax
 \def\ShowGrid{\let\PlaceGrid@\PlaceGrid@@}

 \def\bAdjust@@{%
     \setbox\auxLabelbox@=\hbox{\raise \dp\auxLabelbox@
            \box\auxLabelbox@}}
 \def\bAdjust@{\let\vAdjust@\bAdjust@@}

 \def\tAdjust@@{%
     \setbox\auxLabelbox@=\hbox{\raise-\ht\auxLabelbox@
            \box\auxLabelbox@}}
 \def\tAdjust@{\let\vAdjust@\tAdjust@@}

 \let\vAdjust@\relax

 \def\lAdjust@{\let\hAdjust@\rlap}
 \def\rAdjust@{\let\hAdjust@\llap}

 \let\hAdjust@\relax\let\vAdjust@\relax

 \def\FetchLabel@#1(#2*#3)#4\\#5\endFetchLabel@{%
     \ignorespaces#1\unskip
     \Labeltoks@={#5}%
     \setbox\auxLabelbox@=%
        \hbox to 0pt{\hss\ignorespaces\hAdjust@
          {\ignorespaces#4\unskip}\hss}%
     \vAdjust@
     \let\hAdjust@\relax\let\vAdjust@\relax
     \setbox\Labelbox@=\hbox to 0pt{%
       \box\Labelbox@
       \Shifted@@{#2\Wdd@}{#3\Htt@}{\box\auxLabelbox@}}%
     \ht\Labelbox@=0pt\dp\Labelbox@=0pt
     }

 \def\PlaceLabels@@{\bgroup\def\Cr@{\\}%
     \let\L\lAdjust@\let\R\rAdjust@
     \let\B\bAdjust@\let\T\tAdjust@
     \loop
     \IN@0\Cr@ @\the\Labeltoks@ @\relax
     \ifIN@ \expandafter
       \FetchLabel@\the\Labeltoks@\endFetchLabel@
     \repeat
     \box\Labelbox@\egroup}%

 \let \PlaceLabels@\PlaceLabels@@

 \def\AffixLabels#1{\setbox\Coordinatebox@=\hbox{#1}%
      \Wdd@=\wd\Coordinatebox@ \Htt@=\ht\Coordinatebox@
      \advance\Htt@ \dp \Coordinatebox@
      \hbox{\copy\Coordinatebox@\kern-\Wdd@%
           \Shifted@@{0pt}{-\dp\Coordinatebox@}%
            {\PlaceGrid@\PlaceLabels@}%
           \kern\Wdd@}}

   \let\wlog\LabelFigwlog@ld 
   \catcode`\@=\LabelFigCatAt  

\definecolor{linkred}{rgb}{0.48,0.1,0.05}
\definecolor{linkblue}{RGB}{16, 78, 139}

\definecolor{leafgreen}{rgb}{0.21,0.66,0.24}

\newcommand{\CG}[1]{{\color{leafgreen}#1}}%

\usepackage[hang,flushmargin]{footmisc}
\usepackage{enumitem}
\usepackage{titlesec}
	\titlespacing{\section}{0pt}{12pt}{0pt}
	\titlespacing{\subsection}{0pt}{6pt}{0pt}
	
\titlelabel{\thetitle.\quad}

\makeatletter 

\long\def\@footnotetext#1{%
\H@@footnotetext{%
\ifHy@nesting 
\hyper@@anchor{\@currentHref}{#1}%
\else 
\Hy@raisedlink{\hyper@@anchor{\@currentHref}{\relax}}#1%
\fi 
}}

\def\@footnotemark{%
\leavevmode 
\ifhmode\edef\@x@sf{\the\spacefactor}\nobreak\fi 
\H@refstepcounter{Hfootnote}%
\hyper@makecurrent{Hfootnote}%
\hyper@linkstart{link}{\@currentHref}%
\@makefnmark 
\hyper@linkend 
\ifhmode\spacefactor\@x@sf\fi 
\relax 
}%

\ifFN@multiplefootnote%
\renewcommand*\@footnotemark{%
\leavevmode 
\ifhmode 
\edef\@x@sf{\the\spacefactor}%
\FN@mf@check 
\nobreak 
\fi 
\H@refstepcounter{Hfootnote}%
\hyper@makecurrent{Hfootnote}%
\hyper@linkstart{link}{\@currentHref}%
\@makefnmark 
\hyper@linkend 
\ifFN@pp@towrite 
\FN@pp@writetemp 
\FN@pp@towritefalse 
\fi 
\FN@mf@prepare 
\ifhmode\spacefactor\@x@sf\fi 
\relax%
}%
\fi 

\makeatother 

\theoremstyle{plain}
\newtheorem{theorem}{Theorem}[section]

\newtheorem{lemma}[theorem]{Lemma}
\newtheorem{corollary}[theorem]{Corollary}

\theoremstyle{definition}

\newtheorem{remark}[theorem]{Remark}

\newcommand{\eps}{{\varepsilon}}%
\newcommand{\sys}{{\rm sys}}%
\newcommand{\R}{{\mathbb R}}
\newcommand{\Hyp}{{\mathbb H}}%
\newcommand{\N}{{\mathbb N}}
\newcommand{\CC}{{\mathcal C}}%
\newcommand{\C}{{\mathscr C}}%
\newcommand{\V}{{\mathscr V}}%
\newcommand{\M}{{\mathcal M}}%
\newcommand{\area}{{\rm area}}%
\newcommand{\arcsinh}{{\,\rm arcsinh}}%
\newcommand{\arccot}{{\,\rm arccot}}%
\newcommand{\arccosh}{{\,\rm arccosh}}%
\newcommand{\dist}{\mathrm{dist}}%
\newcommand{\BS}{{\mathcal{BS}}}%

\newcommand{\overfrown}[1]{\rule{-3pt}{0pt}\stackrel{ \frown}{\rule{0pt}{6pt}\smash #1\rule{1pt}{0pt}}\rule{-4pt}{0pt}}

\sectionfont{\large \sc}
\subsectionfont{\normalsize}

\long\def\symbolfootnote[#1]#2{\begingroup%
\def\thefootnote{\fnsymbol{footnote}}\footnote[#1]{#2}\endgroup}

\def\blfootnote{\xdef\@thefnmark{}\@footnotetext}

\begin{document}

{\Large \bfseries \sc Quantifying the sparseness of simple geodesics}

{\Large \bfseries \sc on hyperbolic surfaces}

{\bfseries Peter Buser, Hugo Parlier\symbolfootnote[1]{\normalsize Research partially supported by ANR/FNR project SoS, INTER/ANR/16/11554412/SoS, ANR-17-CE40-0033, by Swiss National Science Foundation grant number PP00P2\textunderscore 153024 and by U.S. National Science Foundation grants DMS 1107452, 1107263, 1107367 RNMS: Geometric structures And Representation varieties (the GEAR Network).
}}

{\em Abstract.} 
The goal of the article is to provide different explicit quantifications of the non density of simple closed geodesics on hyperbolic surfaces. In particular, we show that within any embedded metric disk on a surface, lies a disk of radius only depending on the topology of the surface (and the size of the first embedded disk), which is disjoint from any simple closed geodesic. 
\vspace{1cm}

\section{Introduction}\label{sec:introd}

The set of simple closed geodesics on finitely generated hyperbolic surfaces has many remarkable properties and is related to various aspects of geometric and dynamical properties of moduli spaces and mapping class groups. Among these properties is a result of Birman and Series which states that the larger set of simple complete geodesics is nowhere dense and has Haussdorf dimension 1, the same result holding true for geodesics with uniformly bounded self-intersection number \cite{Birman-Series}. This result, which might seem surprising at first in light of other phenomena, answered a question raised by Jorgensen \cite{Jorgensen} who first exhibited surfaces with non-dense sets of simple closed geodesics. 

It is easy to see that this is really a feature of negative curvature: for instance simple complete geodesics on a flat torus leaving from a given point cover the entire surface and the closed ones are dense (even though they form a measure $0$ subset of the surface). Of course on a flat torus all closed geodesics are simple and similarly, on a hyperbolic surface, the set of {\it all} closed geodesics is not only dense but also dense in the unit tangent bundle. This is one of many instances of how simple geodesics are rare within the set of all closed geodesics. 

Another prime example is for the growth of the number of curves on hyperbolic surfaces. By Huber's asymptotic law \cite{Huber}, the number of closed geodesics of length less than $L$ grows asymptotically like $\sfrac{e^L}{L}$. In contrast, Mirzakhani showed that simple geodesics have very different asymptotic growth as they grow polynomially in $L$ with leading term on the order of $L^{6g-6+n}$ where $g$ is the genus and $n$ the number of cusps \cite{Mirzakhani}. Her results show much more than just asymptotic growth and relate the growth function to the underlying moduli space in many ways. In particular the exact asymptotic behavior doesn't only depend on the topology but also on the underlying geometry of the surface. That the number of simple curves is bounded above by a polynomial function of length was already was one of the key arguments in the results of Birman and Series, and the correct rough asymptotic growth was first proved by Rivin \cite{Rivin1}. More generally, a polynomial upper bound holds for curves with bounded intersection number, and there has been a flurry of results showing asymptoptic growth for such curves in more general contexts \cite{Chas,Erlandsson-Souto, Mirzakhani2, Rivin2, Sapir}.

In this article we quantify the non-density of simple complete geodesics. The Birman-Series result tells us that given any hyperbolic surfaces, in any neighborhood of a point, there is a small disk entirely untouched by any simple complete geodesic. The type of questions we aim to answer are: How large can you take that disk to be? What is the size of the largest disk disjoint from all simple complete geodesics? 

The set of points of $S$ that lie on a simple complete geodesic will be denoted $\BS(S)$ or simply $\BS$ and we shall sometimes refer to this set as the Birman-Series set. Given a surface $S$, we can look at the radius of the largest disk on the complement of $\BS$. Given a moduli space $\M$, the size of this ``maximal gap" is a function over $\M$ and in fact is continuous. Using this and a computation of what happens towards the boundary of moduli spaces, one can show that there is a positive lower bound to $G_S$ which only depends on the topology of the underlying surface (some of the details can be found in \cite{BuserParlier1}). In particular, this shows the existence of a constant $K_{g}>0$ such that any closed hyperbolic surface of genus $g$ has a gap of size $K_{g}$ (or similarly the existence of a constant $K_{g,n}>0$ such that any genus $g$ surface with $n$ cusps has a gap of size $K_{g,n}$). One might hope to find a universal lower bound on the size of gaps but in fact this is impossible. Take any $\varepsilon$-dense but finite set of closed geodesics on a closed surface. Then a theorem of Scott \cite{Scott} asserts that there is a finite cover where all lifts of the closed geodesics are simple. In the cover, the simple geodesics reproduce the $\varepsilon$-density. 

So as there are no universal positive lower bounds on $K_g$, one of our underlying goals is to quantify the constant $K_g$ in terms of $g$. Our first approach to this leads to a precise computation for surfaces in the thin part of moduli space following the natural thick-thin decomposition of moduli space with the {\it systole} function. The {\it systole} $\sys(S)$ of a finite type hyperbolic surface $S$ is the length of a  non-trivial curve of minimal length (by  non-trivial we mean  non-homotopically trivial and  non-peripheral to boundary). Surfaces with systole at least $\varepsilon>0$ are said to lie in the $\varepsilon$-thick part of moduli space. A first step to proving the above theorem is to show that (closed) surfaces with systole below a certain threshold have a gap of that same size.

\begin{theorem}\label{thm:thinsurfaces-i} Let $a_g = \frac{1}{4\cdot (4 \pi (g-1))^2}$. If $\sys(S) \leq a_g$, then $C_S\geq a_g$.
\end{theorem}

The next step is to deal with surfaces with systole length bounded below. To do so we show the following local result which essentially says that given an embedded disk on the surface, there is a quantifiable gap of a certain radius within that disk. The radius only depends on the topology of the underlying surface (and the size of the initial disk of course). 

\begin{theorem}\label{thm:lq1-i} Assume $s = \min\{\frac{1}{2}\sys(S),\frac{1}{3} \}$. Then for any $\rho \leq s$ and any disk $B_{\rho}$ of radius $\rho$ in $S$ there exists a point $p \in B_{\rho}$ such that
\begin{equation*}
\dist(p,\BS) \geq \rho^2 e^{-M(g-1)},
\end{equation*}
where $M$ is an explicit constant that depends only on $s$. 
\end{theorem}

The constant $M$ can be taken to be $\frac{194}{s^2} \log(\frac{134}{s})$ and together with the result for thin surfaces, this provides a quantifiable lower bound on $K_g$. Using the same techniques it also holds for (sufficiently thick) surfaces with cusps. 

Our final goal is to establish a quantified local result which does not depend on systole length. It implies an explicit bound on the constant $K_{g,n}$ discussed previously.

\begin{theorem}\label{thm:lq3-i} Let $S$ be a hyperbolic surface of genus $g$ with $n$ cusps and $B_{\rho}$ a disk of radius $\rho$ in the $\varepsilon$-thick part of $S$, where $0 < \rho < \varepsilon \leq \sfrac{1}{3}$. Then there exists a point $p \in B_{\rho}$ such that
\begin{equation*}
\dist(p,\BS) \geq e^{-3^\kappa R},
\end{equation*}
where $\kappa = 3g-3+n$ and $R = 2 \log \frac{1}{\rho} + M(g-1+\sfrac{n}{2})$ with $M=\frac{195}{\varepsilon^2}\log \frac{134}{\varepsilon}$.
\end{theorem}

Note than in the above, as $B_{\rho}$ always lies in the $\rho$-think part, so in particular the result holds with $\varepsilon = \rho$. Our local estimates without prior assumptions on systole length are weaker. This comes from our method in which we use a sort of classification of short curves: if they are sufficiently short with respect to the other short curves we treat them like cusps and if not we treat them as ``short but not too short curves". 

We end this introduction with some images of the Birman-Series set for genus $2$ surfaces illustrating its intricacy; although these geodesics are nowhere dense, the gaps are already quite small.

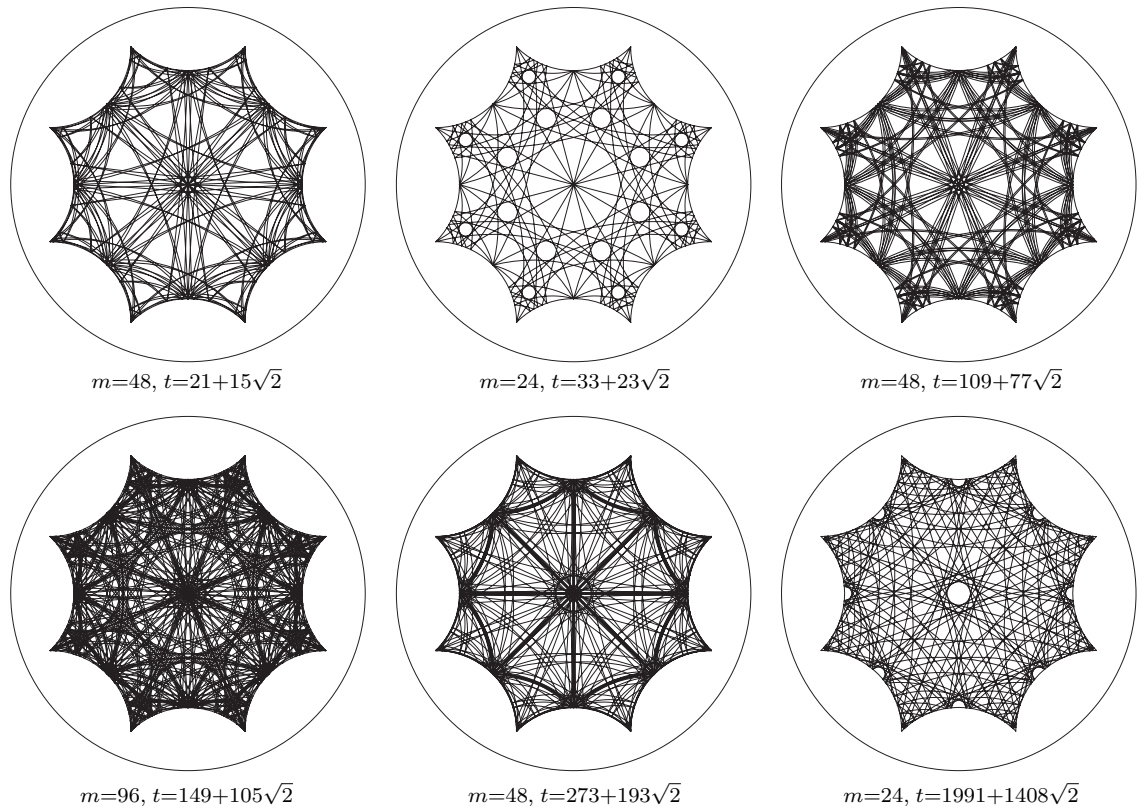
\begin{figure}[h]
 \vspace{0pt}
 \begin{center}
 \leavevmode
 \SetLabels
(0.162*0.50) $\scriptstyle m=48,\; t=21+15\sqrt{2}$\\
(0.50*0.50) $\scriptstyle m=24,\; t=33+23\sqrt{2}$\\
(0.837*0.50) $\scriptstyle m=48,\; t=109+77\sqrt{2}$\\
(0.162*-0.04) $\scriptstyle m=96,\; t=149+105\sqrt{2}$\\
(0.50*-0.04) $\scriptstyle m=48,\; t=273+193\sqrt{2}$\\
(0.837*-0.04) $\scriptstyle m=24,\; t=1991+1408\sqrt{2}$\\
 \endSetLabels
 \AffixLabels{
 \includegraphics[width=15.0cm]{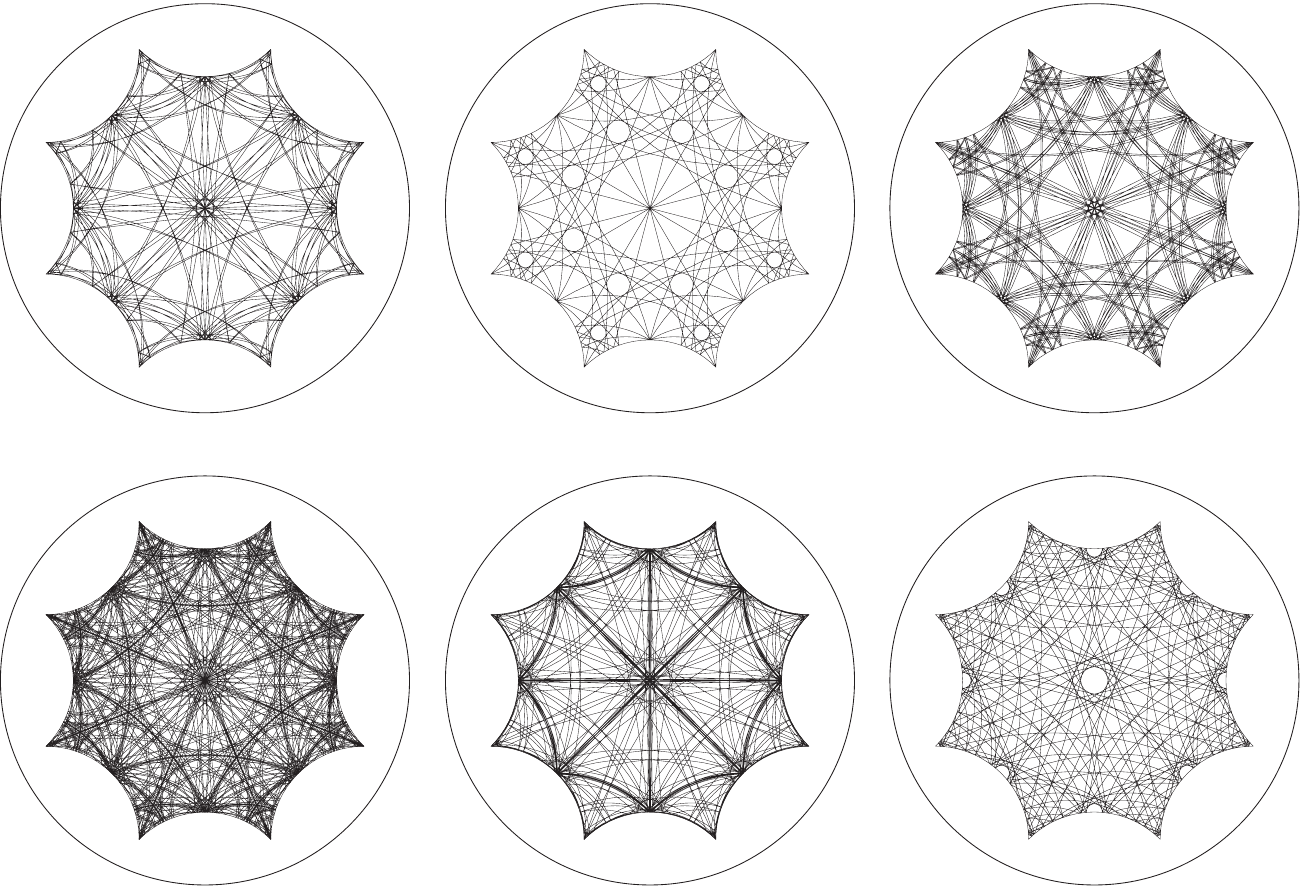} }
 \end{center}
 \vspace{-0pt}
 \caption{\label{Fig:Bolza} Simple closed geodesics on the Bolza surface}
 \end{figure}
 
 Figure \ref{Fig:Bolza} shows a number of simple closed geodesics on the regular fundamental domain of the Bolza surface, the genus 2 Riemann surface with the maximal number of symmetries. Originally the idea was to show all the simple closed geodesics up to length roughly 15 on the same fundamental domain. However, even under optimal printing conditions, the fundamental domain came out evenly black. We have therefore split up the geodesics into families where all members of a family have the same length. Figure \ref{Fig:Bolza} shows a few. In this figure, $m$ is the multiplicity, i.e.\ the number of geodesics in the family, and $t = \arccosh(\sfrac{\ell(\gamma)}{2})$ is half the trace of the conjugacy class in the Fuchsian group that represents a closed geodesic of length $\ell(\gamma)$.

\begin{figure}[t]
 \vspace{0pt}
 \begin{center}
 \leavevmode
 \SetLabels
 \endSetLabels
 \AffixLabels{
 \includegraphics[width=15.0cm]{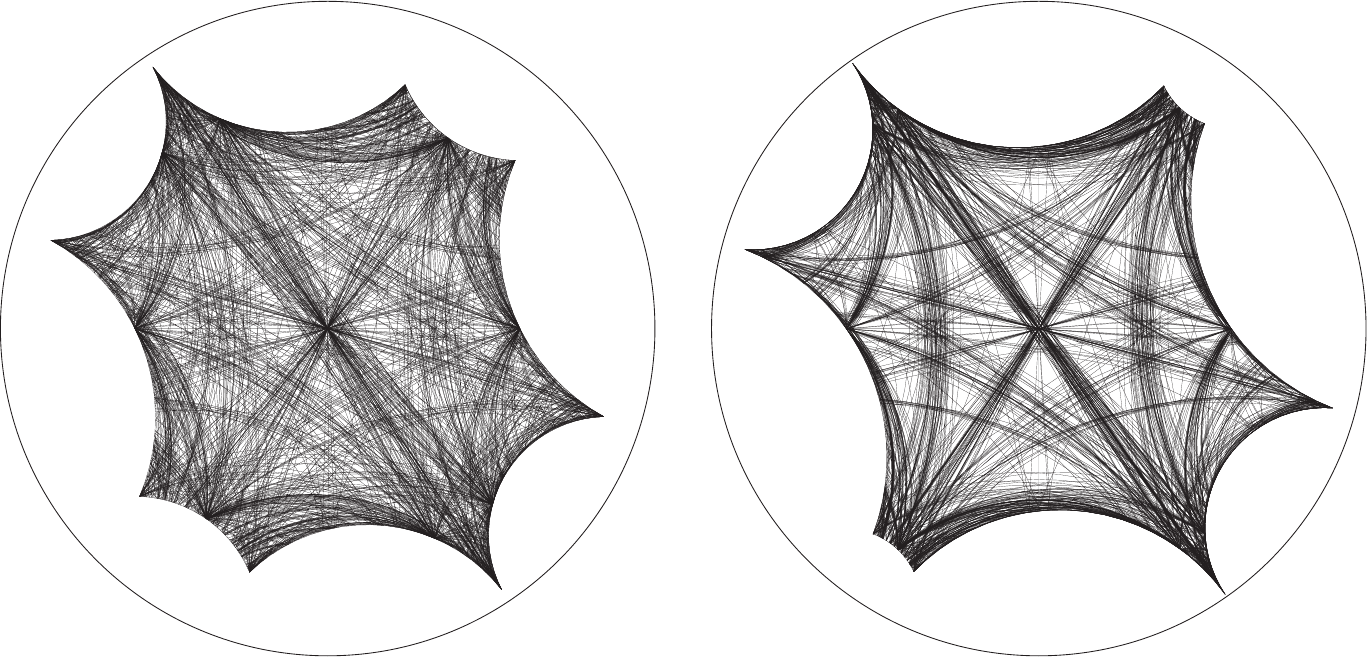} }
 \end{center}
 \vspace{-0pt}
 \caption{\label{Fig:BolzaP} Perturbations of the Bolza surface}
 \end{figure}

Figure \ref{Fig:BolzaP} shows roughly the first (ordered by length) three hundred simple closed geodesics on two other genus 2 surfaces. These surfaces were obtained by perturbing Fenchel-Nielsen length and twist parameters of the Bolza surface. The size of the largest gaps appears to grow. Although we have no real evidence other than these figures, we wonder if the Bolza surface might be the surface where the largest gaps are the smallest. 

{\bf Organization.} The article is organized as follows. In Section 2 we prove Theorem \ref{thm:thinsurfaces-i}, denoted Theorem \ref{thm:thinsurfaces} in the sequel. Section 3 is dedicated to the proof of Theorem \ref{thm:lq1-i}, referred to later as Theorem \ref{thm:lq1}. Thin surfaces and surfaces with cusps are treated in Section 4 where we prove Theorem \ref{thm:lq3-i}, relabelled as Theorem \ref{thm:lq3}. The article is concluded by an appendix which contains two technical results somewhat different in nature from the rest of the article. 

{\bf Acknowledgements.} We heartily thank Chris Judge, Manuel Racle, Klaus-Dieter Semmler and Caroline Series for enlightening conversations and their encouragement. 

\section{Gaps on thin surfaces}\label{sec:gapthn}

In this part we show the following result, where $a_g = \frac{1}{4\cdot (4 \pi (g-1))^2}$.

\begin{theorem}\label{thm:thinsurfaces} If $\sys(S) \leq a_g$, then $C_S\geq a_g$.
\end{theorem}

We shall show explicitly where on the surface a ``forbidden disk'' with the indicated radius may be found. For this we first construct a certain pair of pants on $S$. 
 \begin{lemma}\label{lem:lengthbound}%
Let $\gamma$ be a simple closed geodesic on $S$ of length $\ell(\gamma) < \frac12$. Then there exists a pair of pants $Y \subset S$ with boundary geodesics $\gamma$, $\gamma_1$, $\gamma_2$ such that

 \begin{equation}\label{eq:lengthbound}%
\cosh(\tfrac12 \ell(\gamma_i))< 
\frac{\sinh(\frac12 \ell(\gamma))}{\ell(\gamma)}\cdot 4 \pi (g-1), \quad i=1,2
 \end{equation}%
 \end{lemma}%
Before giving the proof we note that the length bound of $\frac{1}{2}$ on $\gamma$ can be replaced by $\arcsinh (1)$ which is ``optimal" for the argument we give. 

 \begin{proof}%
We adapt an argument from \cite[Section 5.2]{BuserBook}. Cut $S$ open along $\gamma$ into a bordered surface $\tilde{S}$ (consisting of either one or two connected components) with copies $\gamma'$, $\gamma''$ of $\gamma$ on the boundary. Take the component $S'$ of $\tilde{S}$ that has $\gamma'$ on the boundary and look at the sets $C(r) = \{ p \in S' \mid \dist(p,\gamma')\leq r \}$ for $r>0$. For small $r$ this set is an annulus and it does not intersect $\gamma''$. As we let $r$ grow, some first value $r_\gamma$ will be reached where one of these two properties ceases to hold. Now, the injectivity radius near $\gamma''$ is far too small to allow the $C(r)$ to come close to $\gamma''$ as long as $r< r_\gamma$. (Their boundaries are simple closed curves of geodesic curvature smaller than the curvature of a horocycle). Hence, it's the annulus property that ceases to hold for $C(r_\gamma)$.

The rest is exactly as in \cite[Section 5.2]{BuserBook}: There are two geodesic segments of length $r_\gamma$ emanating orthogonally from $\gamma'$ and meeting each other under the angle $\pi$ at their endpoints, thus forming a smooth geodesic arc $\eta$ of length $2r_\gamma$. The endpoints of $\eta$ on $\gamma'$ dissect $\gamma'$ into two arcs $c_1$, $c_2$. The closed curves $c_1 \eta$ and $\eta^{-1}c_2$ are freely homotopic to simple closed geodesics $\gamma_1$, $\gamma_2$ that together with $\gamma'$ form the boundary of a pair of pants $Y \subset S'$ and by {\bf formula \cite[2.3.4(i)]{BuserBook}} the lengths satisfy
$\cosh(\tfrac12 \ell(\gamma_i)) = \sinh(r_\gamma)\sinh(\tfrac12 c_i)< \sinh(r_\gamma)\sinh(\tfrac12 \ell(\gamma))$ for $i=1,2$. Since the interior of $C(r_\gamma)$ is still an annulus we have
$\area(C(r_\gamma)) = \ell(\gamma) \sinh(r_\gamma) < \area(S) = 4\pi(g-1)$ 
and the lemma follows.
\end{proof}%

Throughout the paper we will be using hyperbolic trigonometry. As in the above proof, we will always refer to formula numbers from \cite{BuserBook}. 

We also need an extension of the usual collar lemma \cite[chapter 4]{BuserBook} that includes complete non closed simple geodesics. For any simple closed geodesic $\gamma$ on $S$ the \emph{width} is the quantity
 \begin{equation}\label{eq:width}%
w_\gamma = \arcsinh(1/\sinh(\tfrac12 \ell(\gamma)))
 \end{equation}%
The collar theorem states among other things that the \emph{collar} $\CC_\gamma = \{p \in S \mid \dist(p,\gamma) < w_\gamma \}$ is homeomorphic to an annulus. (In contrast to the above $C(r)$ the collars $\CC_\gamma$ are defined as open sets.) The needed complement is the following.
 \begin{lemma}[Simple geodesics in collars]\label{lem:collar}%
Any complete simple geodesic intersecting $\CC_\gamma$ either intersects $\gamma$ or converges to it.
 \end{lemma}%
 \begin{proof}%
Let $\eta$ be a complete simple geodesic on $S$, closed or non closed, that neither intersects $\gamma$ nor converges to it. Let $p$ be a point on $\eta$ closest to $\gamma$. Let $a$ be a simple geodesic arc from $p$ to $\gamma$ orthogonal to both $\eta$ and $\gamma$ and of length $\ell(a) = \dist(\gamma, \eta) > 0$. In the universal cover of $S$ there are lifts $\tilde{\gamma}$ of $\gamma$ and $\tilde{\gamma_1}$, $\tilde{\gamma_2}$ of $\eta$, as in Figure~\ref{Fig:ExplicitGap} (which we use for two different purposes), together with lifts $a_1$, $a_2$ of $a$ from $\tilde{\gamma_1}$ and $\tilde{\gamma_2}$ to $\tilde{\gamma}$ whose endpoints on $\tilde{\gamma}$ are distance $\ell(\gamma)$ apart from each other. 


Since $\eta$ is simple its lifts $\tilde{\gamma_1}$ and $\tilde{\gamma_2}$ are disjoint; they may have a common endpoint at infinity, though. We thus have a, possibly degenerated, right-angled geodesic hexagon with three consecutive sides of lengths $\ell(a)$, $\ell(\gamma)$, $\ell(a)$. This hexagon splits into two isometric pentagons. Applying {\bf formula \cite[2.3.4(i)]{BuserBook}} to either of them we get
 \begin{equation*}%
\sinh(\ell(a)) \cdot \sinh(\frac12 \ell(\gamma)) \geq 1
 \end{equation*}%
and hence, $\dist(\eta,\gamma) \geq w_\gamma$.
 \end{proof}%


%
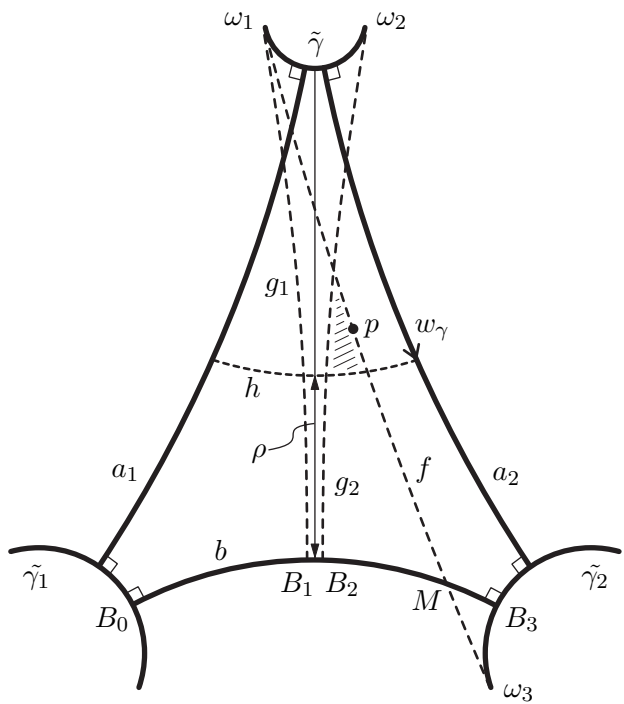
\begin{figure}[ht]
 \vspace{0pt}
 \begin{center}
 \leavevmode
 \SetLabels
 (0.20*0.32) $a_1$\\
 (0.80*0.315) $a_2$\\
 (0.06*0.16) $\tilde{\gamma_1}$\\
 (0.94*0.16) $\tilde{\gamma_2}$\\
 (0.5*0.96) $\tilde{\gamma}$\\
 (0.683*0.538) $w_\gamma$\\
 (0.35*0.195) $b$\\
 (0.38*1.00) $\omega_1$\\
 (0.62*1.00) $\omega_2$\\
 (0.82*-0.01) $\omega_3$\\
 (0.179*0.095) $B_0$\\
 (0.471*0.148) $B_1$\\
 (0.541*0.148) $B_2$\\
 (0.824*0.095) $B_3$\\
 (0.675*0.12) $M$\\
 (0.41*0.35) $\rho$\\
 (0.44*0.6) $g_1$\\
 (0.55*0.3) $g_2$\\
 (0.588*0.538) $p$\\
 (0.67*0.32) $f$\\
 (0.4*0.44) $h$\\
 \endSetLabels
 \AffixLabels{
 \includegraphics{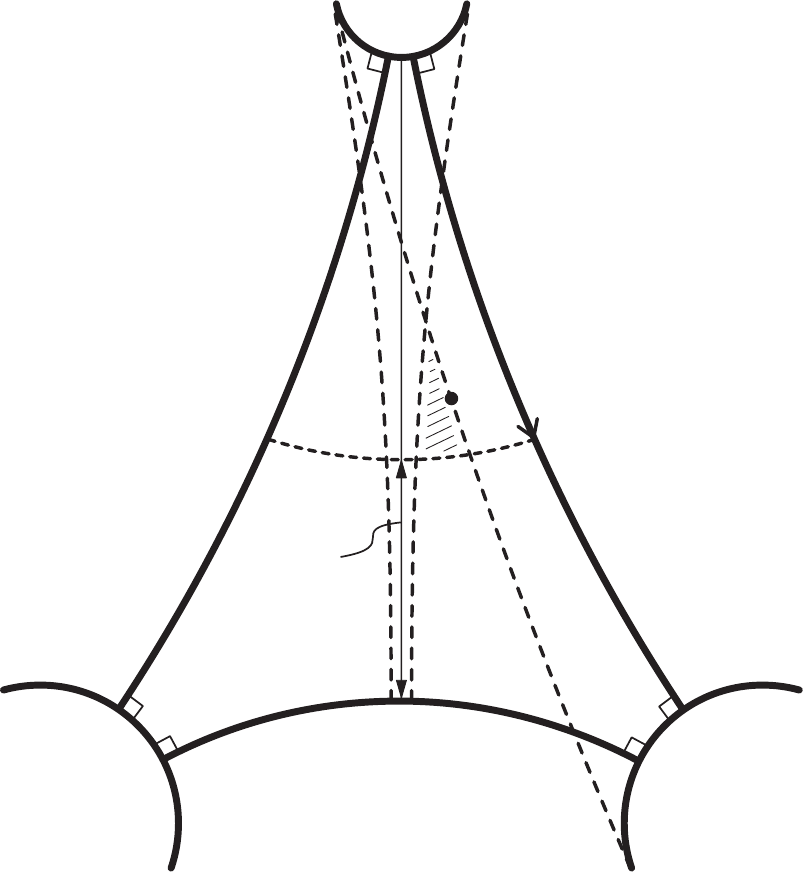} }
 \end{center}
 \vspace{-0pt}
 \caption{\label{Fig:ExplicitGap} The geodesic hexagon $G$ with a forbidden area}
 \end{figure}

 \begin{proof}[Proof of Theorem~\ref{thm:thinsurfaces}]%
Let $\gamma$ be a simple closed geodesic of length $\ell(\gamma) \leq a_g$, where we think of $a_g$ as being rather small; the bound for $a_g$ as in the theorem will show up towards the end of the proof.

By Lemma~\ref{lem:lengthbound} there exist simple closed geodesics $\gamma_1$, $\gamma_2$ with lengths bounded as in \eqref{eq:lengthbound} that together with $\gamma$ form the boundary of a pair of pants $Y$. We shall find a forbidden disk in the half collar $C(w_{\gamma})$ in $Y$.

$Y$ may be decomposed into two isometric right-angled geodesic hexagons by drawing the common perpendicular geodesic arcs $a_1$ from $\gamma_1$ to $\gamma$; $a_2$ from $\gamma$ to $\gamma_2$; and $b$ from $\gamma_2$ to $\gamma_1$. Figure~\ref{Fig:ExplicitGap} shows the lift $G$ of one of the hexagons in the universal covering $\Hyp$ of $S$. (It need not be symmetric as drawn.) The lifts of the three perpendiculars are named after their originals, the lifts of the boundary geodesics are named, respectively, $\tilde{\gamma}_1$, $\tilde{\gamma}$, $\tilde{\gamma}_2$. We let $\omega_1$, $\omega_2$ be the endpoints at infinity of $\tilde{\gamma}$, labelled in such a way that $\omega_1$ and $a_2$ lie on different sides of $a_1$. We also consider the endpoint $\omega_3$ at infinity of $\tilde{\gamma}_2$ that is separated from $a_2$ by $b$.

The dotted lines in Figure~\ref{Fig:ExplicitGap} are as follows. The vertical lines $g_1$, $g_2$ are the perpendicular geodesics from $\omega_1$, $\omega_2$ to $b$ with respective endpoints $B_1$, $B_2$ on $b$. The horizontal line $h$ consists of all points in the hexagon that have distance $w_\gamma$ from $\tilde{\gamma}$. Line $f$ is the geodesic from $\omega_1$ to $\omega_3$.

We will prove below that if $\gamma$ is as short as indicated, then the configuration of these lines is such that a shaded area $\Delta$ occurs, as drawn in the figure, consisting of points in the hexagon lying above $h$, to the right of $g_2$ and to the left of $f$. We will also estimate the size of $\Delta$. Anticipating this for a moment we prove that $\Delta$ is forbidden, i.e., no lift of a simple complete geodesic of $S$ in $\Hyp$ intersects $\Delta$.

Since $\Delta$ lies in a lift of $\CC_\gamma$ Lemma~\ref{lem:collar} implies that only then a simple geodesics $\eta$ may have a lift passing through $\Delta$ when $\eta$ intersect $\gamma$ or converges to it. Now, take any such $\eta$ and let $c \subset \eta$ be a connected component of $\eta \cap \textrm{interior}(Y)$ with one of its ends converging to $\gamma$ (either by converging to a point on $\gamma$ or by winding infinitely often around it). Let $\tilde{c}$ be a component of its lift in $\Hyp$ that meets the part of $G$ above line $h$. We have to show that $\tilde{c} \cap \Delta = \emptyset$. Since $c$ has no self-intersections there are only the following three cases. 

\emph{Case 1:} the other end of $c$ also converges to $\gamma$. Here we consider the geodesic symmetry $\sigma : \Hyp \to \Hyp$ with respect to the geodesic through $b$ and observe that $G \cup \sigma (G)$ is a lift of $Y$. Hence, one end of $\tilde{c} $ converges to $\tilde{\gamma}$ and the other to $\sigma(\tilde{\gamma})$. This implies that the part $\tilde{c} \cap G$ lies between $g_1$ and $g_2$ and cannot intersect $\Delta$.

\emph{Case 2:} the other end of $c$ converges to $\gamma_1$. In this case the other end of $\tilde{c}$ converges to $\tilde{\gamma}_1$. This implies that, speaking with Figure~\ref{Fig:ExplicitGap}, $\tilde{c}$ lies on the left hand side of $g_2$ while $\Delta$ lies on the right hand side.

\emph{Case 3:} the other end of $c$ converges to $\gamma_2$. By the same argument as before $\tilde{c}$ now lies on the right hand side of $f$ and cannot intersect $\Delta$ either. This completes the proof that $\Delta$ is forbidden.

For the existence and the size of $\Delta$ we need a number of estimates beginning with the length of side $b = B_0B_3$. To that end, we abbreviate the right hand side of Inequality \eqref{eq:lengthbound} by $\cosh \lambda$. The inequality in Lemma~\ref{lem:lengthbound} then becomes $\frac12 \ell(\gamma_i) < \lambda$, $(i=1,2)$. By {\bf formula \cite[2.4.1(i)]{BuserBook} }we have the following (where from now on we omit $\ell$ in the formulas):
 \begin{equation*}%
\cosh(\frac12 \gamma) = \cosh(b)\sinh(\tfrac12 \gamma_1)\sinh(\tfrac12 \gamma_2) - \cosh(\tfrac12 \gamma_1)\cosh(\tfrac12 \gamma_2).
 \end{equation*}%
Here the right hand side is $\geq 1$ and it increases when we replace $\frac{1}{2}\gamma_1$ and $\frac{1}{2}\gamma_2$ by $\lambda$. Using the identity $\cosh(b) \sinh^2(\lambda)-\cosh^2(\lambda)=2\sinh^2(\frac{1}{2}b) \sinh^2(\lambda)-1$, we get

 \begin{equation}\label{eq:boundb}%
\sinh(\tfrac12 b) \sinh \lambda \geq 1.
 \end{equation}%
The next estimate concerns the distance $\rho = \dist(b,h)$, where we note that $\rho + w_\gamma$ is the length of the common perpendicular of $b$ and $\tilde{\gamma}$. This perpendicular decomposes $G$ into two right-angled pentagons. Applying {\bf formula \cite[2.3.4(i)]{BuserBook}} to the one that has the bigger side on $\tilde{\gamma}$ we get $\sinh(\rho + w_\gamma) \cdot \sinh(\frac14 \gamma) \leq \cosh\lambda$. Using that $e^\rho \sinh(w_\gamma) \leq \sinh(\rho + w_\gamma)$ and applying the definition \eqref{eq:width} of $w_\gamma$ we get the following, where $\tau = \sinh(\frac12 \gamma)/\sinh(\frac14 \gamma)$ is a factor close to 2:
 \begin{equation}\label{eq:boundrho}%
e^\rho \leq \tau \cdot \cosh \lambda.
 \end{equation}%
We now assume that the labelling of the sides of $G$ has been set such that $\dist(B_1,B_3) > \frac12 b$. Let $M$ be the intersection point of $b$ and $f$, where $f = \omega_1 \omega_3$. Then $M$ is the midpoint of the ideal crossed geodesic quadrilateral $\omega_1 B_1 B_3 \omega_3$ and so we have an ideal right-angled triangle $\omega_1 B_1 M$ with side 

 \begin{equation}\label{eq:B1M}%
B_1M > \frac14 b.
 \end{equation}%
We estimate how far sides $g_1 = \omega_1 B_1$ and $\omega_1 M$ are apart from each other in the neighborhood of $h$. To this end, we take a point $p$ on $f$ and drop the perpendicular $pp_1$ to $g_1$, the position of $p$ being such that $\dist(B_1,p_1) = \rho + r$ with $0 \leq r \leq \frac15$. Applying {\bf formula \cite[2.2.2(iv)]{BuserBook}} to the ideal triangles $\omega_1 B_1 M$ and $\omega_1 p_1 p$ we get, using an obvious limit argument,
 \begin{equation}\label{eq:distp}%
\frac{\tanh(p_1p)}{\tanh(B_1M)}= e^{-(\rho+r)}.
 \end{equation}%
Finally, we provide a similar estimate for $g_2$ and $g_1$. Applying the pentagon {\bf formula \cite[2.3.4.(i)]{BuserBook}} to half of the ideal quadrilateral $\omega_1B_1B_2\omega_2$ we have
 \begin{equation*}\label{eq:distBB}%
\sinh(\tfrac12 B_1B_2)\cdot \sinh(\rho + w_\gamma) = 1.
 \end{equation*}%
Similarly to the preceding step we take a point $q$ on $g_2$ and drop the perpendicular $qq_1$ to $g_1$, the position of $q$ being such that $\dist(B_1,q_1) = \rho + r$ with $0 \leq r \leq \frac15$. By {\bf formula \cite[2.3.1(iv)]{BuserBook}} and using that $\tanh(2t) \leq 2 \sinh( t)$, for $t \geq 0$, we have $\tanh(q_1q) = \cosh(\rho + r)\cdot \tanh(B_1B_2) \leq 2 \cosh(\rho+r) \sinh(\frac12 B_1B_2)$. Hence,

\begin{equation}\label{eq:boundQQ}%
\tanh(q_1q) \leq 2 \frac{\cosh(\rho+r)}{\sinh(\rho+w_\gamma)}\leq 2\frac{\cosh(r)}{\sinh(w_\gamma)} = 2\cosh(r) \sinh(\tfrac12 \gamma).
 \end{equation}%

By \eqref{eq:boundQQ} on the one hand and by \eqref{eq:boundb}--\eqref{eq:distp} on the other, by approximating the small terms in \eqref{eq:boundb}--\eqref{eq:boundQQ} linearly and using that $r \leq \frac{1}{5}$ so that $\cosh(r) < 1.03$ and $e^{-r}> 0.818$, we get:
 \begin{equation*}%
\dist(q_1,q) < 1.1 \,\ell(\gamma), \quad \dist(p_1,p) > \frac{0.8}{(4\pi(g-1))^2}\,.
 \end{equation*}%

It follows that if $\ell(\gamma) \leq a_g$ with $a_g$ as in the statement of the theorem, then in the vicinity of line $h$ the points on $g_2$ are closer to $g_1$ than $1.1 \ell(\gamma)$ and the points on $f$ are further away than $3.2 \ell(\gamma)$. Hence the shaded domain $\Delta$ shows up and one now easily sees that it contains a disk of radius $a_g$.
\end{proof}%

\section{Local quantification in terms of the systole}\label{sec:locqua}

In this section $S$ is a compact hyperbolic surface of genus $g$. We aim to show the following result.

\begin{theorem}\label{thm:lq1} Assume $s = \min\{\frac{1}{2}\sys(S),\frac{1}{3} \}$. Then for any $\rho \leq s$ and any disk $B_{\rho}$ of radius $\rho$ in $S$ there exists a point $p \in B_{\rho}$ such that
\begin{equation*}
\dist(p,\BS) \geq \rho^2 e^{-M(g-1)},
\end{equation*}
where $M$ is a constant that depends only on $s$.
\end{theorem}

We shall get the explicit bound $M = \frac{194}{s^2} \log(\frac{134}{s})$. Observe that Theorem~\ref{thm:lq1} together with Theorem~\ref{thm:thinsurfaces} provides a computable lower bound on the minimum value of $C_S$. 

\subsection{A Voronoi cell decomposition and its properties}\label{sec:Voronoi}

We begin with a lemma that concerns \emph{$\varepsilon$-nets} on $S$, by which we mean a set of points of pairwise distance at least $ \varepsilon$ and maximal for this property with respect to inclusion. The maximality of the set implies that the open balls of radius $\varepsilon$ around the points {\emph cover} the surface. Throughout we shall restrict ourselves to $\varepsilon \leq \frac{s}{2}$, where $s$ is the systole of $S$.

\begin{lemma}\label{lem:netcard}
For fixed $0<\varepsilon\leq \frac{s}{2}$ there exists an $\varepsilon$-net on $S$ consisting of points $\{p_i\}_{i=1}^N$ with
 \begin{equation*}%
N\leq \frac{2}{\cosh \sfrac{ \varepsilon }{2}- 1} (g-1).
 \end{equation*}%
\end{lemma}
\begin{proof}
To construct such a set begin with an arbitrary point and add successively new points at least $\varepsilon$ away from the preceding ones until this is no longer possible. The open balls of radius $\varepsilon/2$ around the resulting points $p_i$ are pairwise disjoint, embedded and of area $2\pi (\cosh \sfrac{ \varepsilon }{2} -1)$. The total area of the surface is $4\pi(g-1)$ thus the number of balls cannot be greater than
 \begin{equation*}%
\frac{4\pi(g-1)}{2\pi (\cosh \sfrac{ \varepsilon }{2} -1)}
 \end{equation*}%
which proves the lemma.
\end{proof}

Given an $\varepsilon$-net one gets a cell decomposition of the surface given by the Voronoi cells associated to each point: each open cell $V_i$ is the set of points whose closest point in the net is $p_i$. The boundary of $V_i$ consist of points closest to more than one of the $p_k$s. By the maximality of the $\varepsilon$-net any boundary point of $V_i$ is closer to $p_i$ than $\varepsilon$. Since $\varepsilon \leq \frac{s}{2}$ it follows that $V_i$ is contained in an embedded disc. Thus, the $V_i$ are simple convex hyperbolic polygons. Since the vertices lie at distance $\leq \varepsilon$ from the center the sides of $V_i$ have lengths $\leq 2 \varepsilon$. 
 
\begin{remark}\label{rem:perturbation}
Note that up until now we haven't given any restrictions on how we choose the points. By standard perturbation techniques, if the points are chosen ``generically", there will be exactly three cells adjacent to any vertex. By this we mean that by using a standard measure on the choice of points (using the measure on the surface), the choices where all vertices are {\it not} adjacent to exactly three cells lie in a measure $0$ set. In all that follows, we shall suppose that this is the case. Hence, there will always be a triangulation dual to our Voronoi cell decompositions. 
\end{remark}

\begin{lemma}\label{lem:vornum} Each Voronoi cell has at most $v(\varepsilon)$ sides, where $v(\varepsilon)$ is the integer part of 
 \begin{equation*}%
\pi / \arccot\left( \cosh(\eps) \cdot\left \{ \sqrt{1 + 2 \cosh{\eps}} + \sqrt{2 + 2 \cosh{\eps}} \right\}\right).
 \end{equation*}%
\end{lemma}

\begin{proof}
Consider the triangulation dual to the Voronoi cell decomposition. Each side of a triangle has length at least $\varepsilon$. Any of the triangles is contained in a ball of radius $\varepsilon$ (around the intersection point of the three Voronoi edges dual to the triangle). The number of sides of a Voronoi cell is the number of dual triangles that meet in its center. We now apply Lemma~\ref{lem:MinimalAngle} from the appendix.
\end{proof}

From the two previous results we deduce a bound on the number of edges found in our Voronoi cell decomposition. We shall, however, assume from now on that $\varepsilon$ is small. A restriction that has proved to be practical is $\varepsilon \leq \frac{1}{3}$.
\begin{corollary}\label{cor:edges}
Assuming that $\varepsilon \leq \min\{\frac{s}{2},\frac{1}{3} \}$ we have $v(\varepsilon) = 12$, and there are at most 
 \begin{equation*}%
12\frac{g-1}{\cosh \sfrac{ \varepsilon }{2}- 1} < \left(\frac{97}{\varepsilon^2}-10\right)(g-1)
 \end{equation*}%
edges in a cell decomposition obtained as above.
\end{corollary}
\begin{proof}
The number of cells times the maximum number of edges per cell is an upper bound on twice the number of edges. Now $v(\varepsilon)$ from Lemma~\ref{lem:vornum} is monotone increasing and $v(0) = v(\frac{1}{3}) = 12$. This proves the result.
\end{proof}

In everything that follows in this section we shall suppose that we have a fixed Voronoi cell decomposition $\V$ like the one we have just constructed with $\varepsilon \leq \min\{\frac{s}{2},\frac{1}{3} \}$.

\subsection{Scheme of proof}\label{sec:scheme}

In this subsection we describe the strategy of our proof without any computations in order to motivate the estimates we perform afterwards.

Given an embedded disk $D = B_{\rho}$ of radius $\rho>0$ on $S$, we look at the restriction of the Birman-Series set to $D$. This consists of a countable collection of geodesic segments between points on the boundary of $D$. 

To find quantifiable ``empty space'' between these segments we begin by introducing a constant $R>0$ that later will be adjusted. For each segment $c$ we look at a larger geodesic segment $\bar{c}$ obtained by extending $c$ outside of $D$. Specifically, we take the continuation of $c$ of length $R-2\eps$ in both directions (measured from the midpoint of $c$). This larger ``segment" is not necessarily embedded but as $c$ belongs to the Birman-Series set, it does not contain any transversal self-intersections. Either endpoint of $\bar{c}$ lies in some Voronoi cell of $\V$. (The two cells may coincide.)

To ``capture'' $\bar{c}$ we associate to it what we shall call a {\it model strand} $m^c$ the precise construction of which shall follow in the next subsection. Among its properties we have the following. The endpoints of $m^c$ lie on vertices of the $1$-skeleton of $\V$ and $m^c$ is freely homotopic to $\bar{c}$, where one allows endpoints to move within the start and terminal Voronoi cells of $\bar{c}$. The important part of $m^c$ is its middle intersection with $D$, i.e.,\ the connected component $m_c$ of $m^c \cap D$ that contains the midpoint of $m^c$. We call $m_c$ a \emph{model strand in} $D$. 

The model strands shall allow us to quantify the arguments in \cite{Birman-Series} by showing that the number of all $m_c$ grows polynomially with $R$ while the distance between $c$ and $m_c$ has an upper bound that decays exponentially with $R$. As a result, for large enough $R$ there are ``relatively few'' model strands in $D$ and each segment $c$ lies in an ``extremely small'' tubular neighborhood of some of these. Outside these tubular neighborhoods we shall then have empty space whose size can be estimated in terms of $R$.

The model strands will be obtained in two steps. First we associate to $\bar{c}$ a combinatorial path $P(\bar{c})$ on the $1$-skeleton of $\V$ that is simple in the combinatorial sense (see the next subsection ). The number of such paths can be estimated using combinatorial arguments. In the second step we then homotope the path $P(\bar{c})$ into a geodesic arc $m^c$ keeping the endpoints fixed. It may turn out that $m^c$ has self-intersections, but $m_c$ is simple. The distance between $c$ and $m_c$ is estimated in Section~\ref{sec:bounddist}.

\subsection{Constructing combinatorial arcs and their properties}\label{sec.concom}

Here we describe a method for associating to any simple arc such as $\bar{c}$ a combinatorial path that lies in the $1$-skeleton of $\V$.

Consider a finite oriented geodesic arc $a$ on $S$ which we assume to be simple in the sense that it has no transversal self-intersctions. We also assume that $a$ is longer than $4 \varepsilon$ (in our applications $a$ will be much larger). Then $a$ begins in some Voronoi cell $V_0$, traverses a sequence of cells $\{V_i\}_{i = 1}^n$ and ends up in some cell $V_{n+1}$ (each cell may appear multiple times). We break up $a$ into smaller arcs $a_i$ where for $i = 0, \dots, n+1$ each $a_i$ is a connected component of $a \cap V_i$ and $a_i$ is connected to $a_{i+1}$ ($i \leq n$). Intersections that consist of a vertex only are ignored, that is we list the cells such that each $a_i$ has positive length. For the combinatorial path we proceed as follows.

\begin{itemize} \item $a_0$ and $a_{n+1}$ do not contribute to our combinatorial path.
\item For $i = 1, \dots, n$, if $a_i$ intersects the interior of $V_i$ we associate to it the homotopic polygonal arc $P_i$ on the boundary $\partial V_i$ that has the same endpoints as $a_i$ and uses the minimal number of edges of $V_i$ (see Figure~\ref{fig:PolygonPath}). 
\item If $a_i$ is a side of $V_i$ we set $P_i$ to be $a_i$. 
\end{itemize}

We add a refinement to item 2: if $V_i$ has an even number of sides, say $m_i$, and if $a_i$ connects a pair of opposite sides of $V_i$, then there are two choices for the path $P_i$, either using $\frac{1}{2}m_i + 1$ edges of $V_i$. To remove this ambiguity we choose, for any Voronoi cell $V$, a ``separator point'' $p_V$ (not necessarily the geometric center of $V$) in the interior of $V \setminus a$ and then require that in the above cases $P_i$ is the path homotopic to $a_i$ in the punctured disk $V_i \setminus \{ p_{V_i} \}$. 

\begin{figure}[ht]
\leavevmode \SetLabels
\endSetLabels
\begin{center}
\AffixLabels{\centerline{\epsfig{file =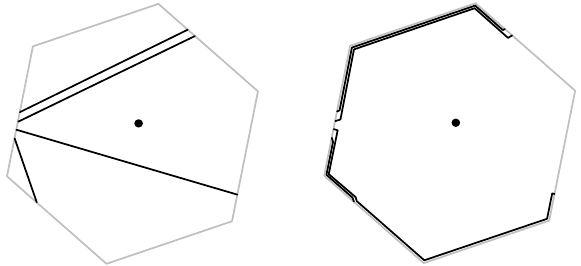,width=6cm,angle=0}}}
\vspace{-24pt}
\end{center}
\caption{\label{fig:PushPath}Pushing the path to the $1$-skeleton}
\end{figure}

The concatenation of these paths, $P_1 \cup \dots \cup P_n$, yields a connected path that lies entirely on the $1$-skeleton of $\V$, possibly with partial edges. We shall remove the latter by ``shrinking'' homotopies that take place on the edges of the $1$-skeleton as indicated in Figure~\ref{fig:PartialElim} so as to obtain a purely combinatorial connected sequence of edges $P(a)$ of the Voronoi polygons. 

\begin{figure}[ht]
\leavevmode \SetLabels
\endSetLabels
\begin{center}
\AffixLabels{\centerline{\epsfig{file =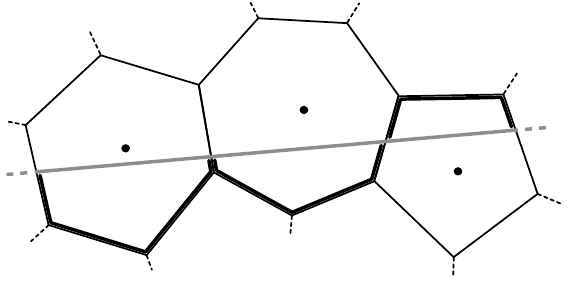,width=7.0cm,angle=0}}}
\vspace{-24pt}
\end{center}
\caption{\label{fig:PolygonPath}Obtaining a polygonal path}
\end{figure}

Thus, the construction is in two steps, the first step consisting in ``pushing the arcs to the boundary'' as indicated in Figure~\ref{fig:PushPath}, the second step being the shrinking away of the partial edges. 

The combinatorial path $P(a)$ obtained in this way is simple in that it does not contain any transversal self-intersections. To see this one my consider an arbitrarily thin tubular neighborhood $T$ around the $1$-skeleton and slightly modify the pushing and shrinking homotopies so as to obtain a genuine simple path $P'(a)$ in $T$ with the same combinatorics as $P(a)$.

\begin{figure}[ht]
\leavevmode \SetLabels
\endSetLabels
\begin{center}
\AffixLabels{\centerline{\epsfig{file =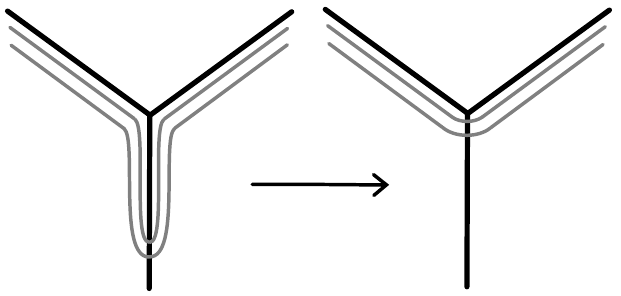,width=5.0cm,angle=0}}}\vspace{-24pt}
\end{center}
\caption{ \label{fig:PartialElim}Elimination of partial edges by shrinking}
\end{figure}

We define the \emph{combinatorial length} of $P(a)$ to be the number of edges it contains. The next lemma provides an upper bound of this length in terms of the original path length.
\begin{lemma}\label{lem:comblength}
A path $a$ as above of length $\ell$ passes through at most 
 \begin{equation*}%
\frac{4}{\varepsilon}(\ell + 3\varepsilon)
 \end{equation*}%
Voronoi cells. This multiplied by 6 is an upper bound for the combinatorial length of $P(a)$.
\end{lemma}

\begin{proof}
We argue in the universal cover by considering a lift $\tilde{a}$ of $a$. Along $\tilde{a}$ we have the sequence of lifts $\tilde{V}_1, \dots, \tilde{V}_{n}$ of the Voronoi cells $V_1, \dots, V_{n}$ with the property that $\tilde{V}_i \cap \tilde{a}$ is a lift of $a_i$, $i=1, \dots, n$ (recall that $a_0$ and $a_{n+1}$ are not taken into account). The interiors of the lifted cells are pairwise disjoint and so are the balls of radius $\sfrac{\varepsilon}{2}$ around their centers $q_1, \dots, q_{n}$. Furthermore, these balls lie in a tubular neighborhood of radius $\sfrac{3\varepsilon}{2}$ around $\tilde{a}$. We now apply an area argument. The total area of the $\sfrac{\varepsilon}{2}$-balls around the points $\{q_i\}_{i=1}^n$ is 
 \begin{equation*}%
n \cdot 2\pi (\cosh \sfrac{\varepsilon}{2}-1).
 \end{equation*}%
The area of the $\sfrac{3\varepsilon}{2}$ neighborhood of $\tilde{a}$ is
 \begin{equation*}%
2\left(\ell \sinh\sfrac{3\varepsilon}{2} + \pi(\cosh \sfrac{3\varepsilon}{2}-1)\right).
 \end{equation*}%
Area comparison and elementary simplification using that $\varepsilon \leq \frac{1}{3}$ now yields the first statement. The second statement follows from Lemma~\ref{lem:vornum} using that $v(\varepsilon) \leq v(\frac{1}{3}) =12$ and that for the crossing of any cell $V_i$ the combinatorial path $P(a)$ runs along at most half the edges of $V_i$.
\end{proof}

\begin{remark} Although this is not needed in what follows we indicate a reverse inequality giving a lower bound on the number of Voronoi cells a path $a$ of length $\ell$ as in the lemma necessarily traverses. To see this consider a maximal set of points $2\varepsilon$ apart on the covering $\tilde{a}$. There are at least $\sfrac{\ell}{2 \varepsilon}$ such points. As Voronoi cells have diameters less than $2 \varepsilon$ this means that $\tilde{a}$, respectively $a$ needs to traverse at least 
 \begin{equation*}%
\frac{\ell}{2\varepsilon}-1
 \end{equation*}%
Voronoi cells. There is also a lower bound on the combinatorial length $L$ of $P(a)$: using that the edges of a Voronoi cell are not longer than the diameter the triangle inequality yields $2 \varepsilon (L+2) \geq \ell$ and thus
 \begin{equation*}%
L \geq \frac{\ell}{2\varepsilon}-2.
 \end{equation*}%

\end{remark}

\subsection{Counting combinatorial paths}\label{S.count}

We need to bound the number of combinatorial paths in terms of the lengths of the model strands. Using the lemma above, this can be obtained via the combinatorial lengths. For the following we denote by $E$ the number of edges of our Voronoi decomposition $\V$.

\begin{lemma} \label{lem:combcount} There are at most $4L^2\binom{L+E}{L}$ combinatorial simple paths of length at most $L$.
\end{lemma}

\begin{proof}
The model for the argument that follows is the following useful fact: the homotopy class of a (simple) multicurve lying on a triangulated surface is determined by its intersection numbers with each of the sides of the triangulation. The triangulation in our case is the triangulation dual to the Voronoi cells and the intersection numbers are exactly the numbers of times the path traverses a given edge. So if our path was closed it would be uniquely determined by the number of times it passes through each edge.

\begin{figure}[ht]
\leavevmode \SetLabels
\L(.47*1.02) $a=3$\\
\L(.32*.01) $b=5$\\
\L(.63*.01) $c=4$\\
\endSetLabels
\begin{center}
\AffixLabels{\centerline{\epsfig{file =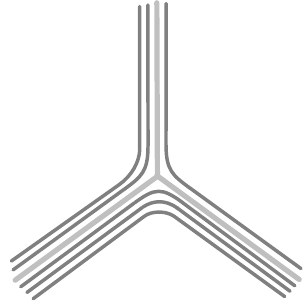,width=4.0cm,angle=0}}}\vspace{-24pt}
\end{center}
\caption{\label{fig:triangles}Example of how the numbers of parallel segments determine the combinatorial paths} 
\end{figure}

We now modify this for non closed paths. In order not to confuse path edges with the $E$ edges of the $1$-skeleton of the Voronoi decomposition, we shall call path edges \emph{segments}. 

We begin by distributing a number $\leq L$ of segments (later to be concatenated) among the $E$ edges. By elementary combinatorics this is possible in $\binom{L+E}{L}$ different ways. 
For ease of description we place them as distinct segments parallel to the corresponding edges in a thin tubular neighborhood of the $1$-skeleton of $\V$ as drawn in Figures~\ref{fig:triangles} and \ref{fig:trianglestwo}. 

Next we select, for any such distribution, a pair of segments that shall play the role of the end segments of the path to be constructed and erase half of each of the two segments. This is possible in at most $4L^2$ different ways. Figure~\ref{fig:trianglestwo} shows an example with half a segment erased.


\begin{figure}[ht]
\begin{center}
\AffixLabels{\centerline{\epsfig{file =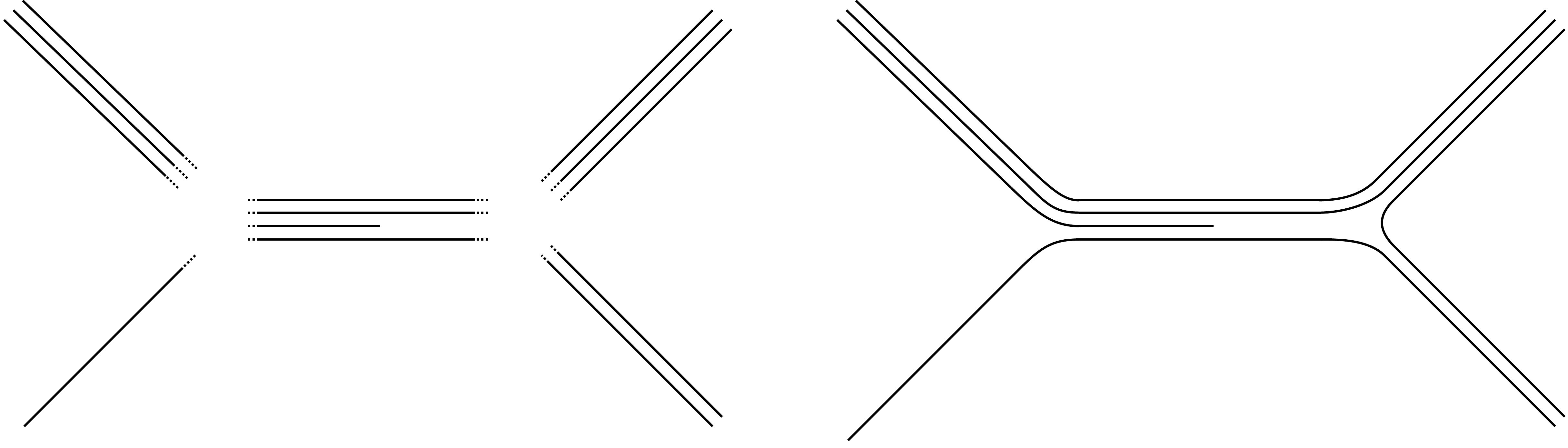,width=10.0cm,angle=0}}}
\vspace{-24pt}
\end{center}
\caption{\label{fig:trianglestwo}An example of how to construct the path with a given starting segment} 
\end{figure}

At any vertex of the $1$-skeleton there are now a number of incoming segments and there is at most one way to paste these together at the endpoints such that the resulting arcs do not intersect each other (pairs of segments on the same edge are not allowed to be pasted together). Of course, not every distribution allows one to paste all segments at all vertices, and even if the pasting is possible the result may be disconnected; but every simple path of length $\leq L$ may be obtained by some distribution and, hence, there are at most $4L^2 \binom{L+E}{L}$ such paths. 
\end{proof}

\subsection{Bounding the distance between arcs and model arcs}\label{sec:bounddist}
%

In what follows, we need to metrically compare the model arcs to the arcs they are intended to approximate. This is achieved via the following lemma about arcs in the hyperbolic plane.

\begin{lemma}\label{lem:width}
Fix $\delta>0$, $\rho>0$. In the hyperbolic plane, let $b,b'$ be two geodesic arcs of lengths $> 2(\rho + \delta)$ such that the two initial points and the two endpoints are at respective distances $\leq \delta$ from each other. Furthermore, consider a disk $D$ of radius $\rho$ centered at the midpoint of $b$.
Then $D\cap b$ lies in an $r$-neighborhood of $b'$, where

 \begin{equation*}%
r\leq \arcsinh\left( \frac{\cosh \rho \sinh \delta}{\cosh \frac{\ell(b)}{2}}\right).
 \end{equation*}%

\end{lemma}

\begin{proof}
We search for an ``extremal" $b'$ that shall allow us to compute the constants appearing in a worst case scenario. For this we consider the two disks of radius $\delta$ surrounding the endpoints of $b$.

An extremal $b'$ must have its endpoints on the boundary of these disks and a moment's reflection shows that the worst case scenario is given by the two geodesics tangent to the boundary of these disks that do not cross $b$ (see Figure~\ref{fig:quad}).

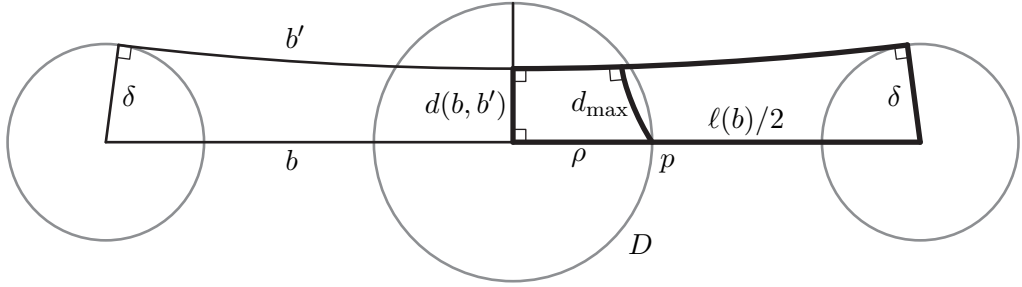
\begin{figure}[ht]
\leavevmode \SetLabels
\L(.30*.397) $b$\\
\L(.30*.83) $b'$\\
\L(.55*.43) $\rho$\\
\L(.628*.41) $p$\\
\L(.157*.625) $\delta$\\
\L(.422*.60) $d(b,b')$\\
\L(.55*.605) $d_{\max}$\\
\L(.67*.545) $\ell(b)/2$\\
\L(.827*.625) $\delta$\\
\L(.60*.10) $D$\\
\endSetLabels
\begin{center}
\AffixLabels{\centerline{\epsfig{file =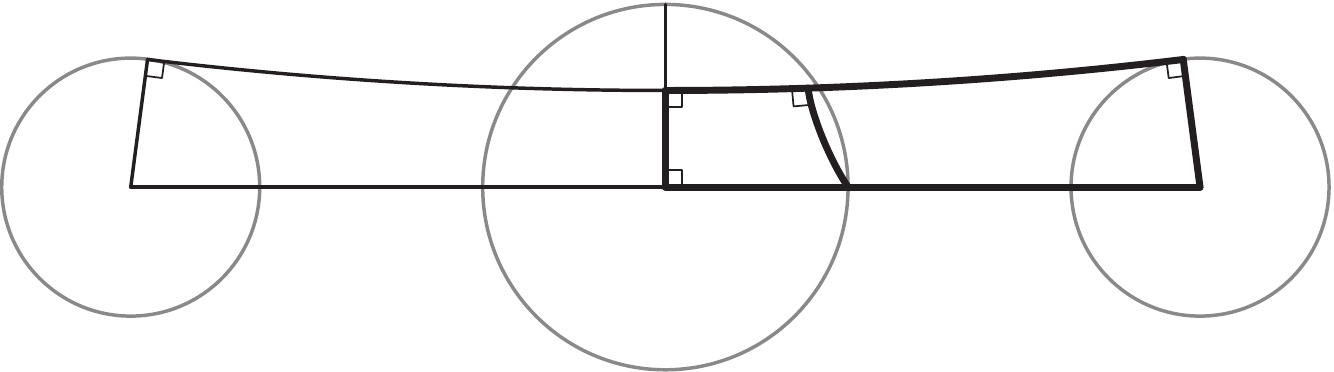,width=13.5cm,angle=0}}}
\vspace{-24pt}
\end{center}
\caption{\label{fig:quad}Comparing the two arcs}
\end{figure}

Let $b'$ be one of them. By symmetry, the distance path between $b$ and $b'$ reaches $b$ at its midpoint. One can now compute in the trirectangles as shown in the figure. The bigger trirectangle has sides of lengths $\delta$, $\frac{1}{2} \ell(b')$, $d(b,b')$, $\frac{1}{2} \ell(b)$. Using hyperbolic trigonometry {\bf formula \cite[2.3.1(v)]{BuserBook}}) we obtain
 \begin{equation*}%
\sinh \delta = \sinh d(b,b') \cosh\frac{\ell(b)}{2}.
 \end{equation*}%
%

%
The least upper bound $d_{\max}$ for the distances to $b'$ of points on $D \cap b$ is reached for the point $p$ where $b$ intersects the boundary of $D$. It remains to show that for $r := d_{\max}$ the inequality as in the lemma is satisfied. We actually have equality: $p$ is the vertex of the smaller trirectangle with sides $d_{\max}$, $\rho'$, $d(b,b')$, $\rho$; using the same formula as before we obtain
 \begin{equation*}%
\sinh d_{\max} = \sinh d(b,b')\cosh\rho
 \end{equation*}%

and putting the two formulas together we get 
 \begin{equation*}%
\sinh d_{\max} = \frac{\cosh \rho \sinh \delta}{\cosh \frac{\ell(b)}{2}}.
 \end{equation*}%

\end{proof}

\subsection{Estimates and finalizing the proof}\label{sec:finaliz} 

We first collect and simplify some of the earlier bounds.

{\it Number of edges.}\;
In view of Corollary~\ref{cor:edges}, we use the abbreviations
\begin{equation}\label{eq:defG}
G :=\frac{97(g-1)}{\varepsilon^2}, \quad G' := G-10.
\end{equation}
By the same corollary, $\V$ has at most $G'$ edges.

{\it Combinatorial length.}\; As described in Section~\ref{sec:scheme} any arc $c$ in the disk $D = B_{\rho}$ that occurs as a connected component of $D \cap \gamma$ for some simple complete geodesic $\gamma$ on $S$ is extended to a larger arc $\bar{c}$ on $\gamma$. The extension goes in both directions, starting from the midpoint of $c$, each half of the extension having length $R - 2\varepsilon$. To $\bar{c}$ we associate the combinatorial path $P(\bar{c})$ as described in Section~\ref{sec.concom} (with $\bar{c}$ in the role of $a$). The extension $\bar{c}$ has length $\ell = 2R-4\varepsilon$ and by Lemma~\ref{lem:comblength} the combinatorial length of $P(\bar{c})$ is bounded above by 
\begin{equation}\label{eq:defL}
L_R:=\frac{48}{\varepsilon}R.
\end{equation}

{\it The number of model strands.}\; By Lemma~\ref{lem:combcount} there are at most 
\begin{equation}\label{eq:bdLGG}
4L^2\binom{L+G'}{L} \leq \frac{4L^2 (L+G')^{G'}}{G'!} \leq \frac{4G^2 (L+G)^G}{G!}
\end{equation}
combinatorial paths $P(\bar{c})$ of combinatorial length $L$. For $L=L_R$ this is at the same time an upper bound for the number of model strands $m^c$, respectively the number of strands $m_c$ in $D$. \linebreak A heuristic check shows that the area argument that will follow further down can only succeed if 
 \begin{equation*}%
R > G. 
 \end{equation*}%
We shall therefore work from now on under this hypothesis. Using Stirling's formula and the fact that $G\geq \frac{97}{\varepsilon^2}$ we then get the following bound for the number of model strands which has been tailored in view of its later application:
\begin{equation}\label{eq:defNR}
N(R):= \frac{1}{10} \frac{m^G}{G^G}R^G \quad \text{with}\quad m = \frac{134}{\varepsilon}.
\end{equation}

{\it Distance between $c$ and $m_c$.}\; By construction the two initial, respectively endpoints of $\bar{c}$ and its associated model strand $m^c$ lie in the same Voronoi cells and their respective distances are smaller than $2 \varepsilon$. Furthermore, 
$\bar{c}$ has length $2R - 4\varepsilon$. By Lemma~\ref{lem:width}, $c$ lies in an $r$-neighborhood of $m_c$ where, by elementary simplification and using that $\rho \leq \varepsilon \leq \frac{1}{3}$,
\begin{equation}\label{eq:tubular}
r \leq \frac{\cosh \rho \sinh 2\varepsilon}{\cosh(R-2\varepsilon)} \leq 3 e^{-R} =: w_R.
\end{equation}

{\it Area argument.}\; This is the heart of the proof. For subsets $A \subset D$ and $t >0$ we shall denote by $A^t$ the part of the $t$-neighborhood of $A$ that lies in $D$. For each model strand $m_c$ in $D$ the set $m_c^{2w_R}$ has area
\begin{equation}\label{eq:areamcwR}
\area(m_c^{2w_R}) \leq 2 \ell(m_c)\sinh(2w_R)\leq 4\rho \sinh(2w_R) < 9\, \rho\, w_R.
\end{equation}
Now let $\mathscr{M}$ be the union of all model strands in $D$. Then $\mathscr{M}^{2w_R}$ has area less than $9 \rho\, w_R $ times the number of model strands in $D$, while $D$ has area $2\pi( \cosh(\rho)-1) > \pi \rho^2$. On the other hand $\BS$ is contained in $\mathscr{M}^{w_R}$. Hence, if we can determine $R$ such that the area bound for $\mathscr{M}^{2w_R}$ is smaller than $\pi \rho^2$, then $\mathscr{M}^{2w_R}$ does not cover $D$. By \eqref{eq:tubular} we therefore get a point $p \in D$ at distance $w_R$ from $\BS$ if we take $R$ to be a solution $>G$ to the equation
\begin{equation}\label{eq:theequation}
\frac{1}{\rho}\frac{m^G}{G^G}R^G = e^R.\end{equation}

{\it Estimating the solution and end of proof.}\; We first state the following lemma which shall also be used in the next section.
\begin{lemma}\label{lem:solutionbound}
Let $\alpha$, $\gamma$ be positive constants, $\alpha > \frac{e^\gamma}{\gamma^\gamma}$. Then the equation $\alpha t^\gamma = e^t$ for $t>0$ has two solutions $t_1 < \gamma$, $t_2 > \gamma$ and
\begin{equation*}
\log(\alpha \gamma^\gamma) < t_2 < 2\log(\alpha \gamma^\gamma).
\end{equation*}
\end{lemma}
\begin{proof}
The equation is equivalent to $\gamma \alpha^{\frac{1}{\gamma}} \frac{t}{\gamma} = e^{\frac{t}{\gamma}}$. Substituting $\tau = \frac{t}{\gamma}$ we transform it into
\begin{equation}\label{eq:prooflemma}
 \frac{e^\tau}{\tau}=\beta
\end{equation}
for $\tau>0$ with $\beta=\gamma \alpha^{\frac{1}{\gamma}}$. By the hypothesis on $\alpha$ we have $\beta >e$ and so \eqref{eq:prooflemma} has two solutions $\tau_1 < 1$ and $\tau_2 > 1$. For $\tau' = \log{\beta}$ we have $\tau' > 1$ and $\frac{e^{\tau'}}{\tau'}<\beta$. For $\tau''=2\log{\beta}$ we have $\tau'' < \beta$ and $\frac{e^{\tau''}}{\tau''}>\beta$. Hence, $\log{\beta} < \tau_2 < 2\log{\beta}$. Substituting back we get the claims of the lemma.
\end{proof}
In the case of equation \eqref{eq:theequation} we have $\gamma = G$ and $\alpha = \frac{1}{\rho} \frac{m^G}{G^G}$. By Lemma~\ref{lem:solutionbound} the larger of the two solutions to the equation has the bound $R \leq 2 \log(\frac{1}{\rho}) + 2G \log(m)$. We thus get the lower bound on the maximal distance to $\BS$ in $D$: 
\begin{equation}\label{eq.final bound}
w_R \geq 3\rho^2 e^{-2G \log{m}}.
\end{equation}
For $\varepsilon = s$ this yields the bound as stated in Theorem~\ref{thm:lq1}.

\section{Surfaces with cusps and (very) small geodesics}\label{sec:smallcusps}

In this section we generalize the local quantification results of the previous sections to include surfaces with cusps and small geodesics. The process will be almost identical to previously with just a number of necessary changes that will be detailed. Again $\varepsilon$ is a fixed constant, $0 < \varepsilon \leq \frac{1}{3}$.

\subsection{Thick and thin decomposition}\label{sec:thickthin} We consider this time a complete orientable finite area hyperbolic surface $S$ of genus $g$ with $n \geq 0$ cusps. A closed geodesic on $S$ shall be called \emph{small} if its length is $\leq 2\varepsilon$. We let $\beta_1, \dots, \beta_h$ be the list of all small geodesics of $S$ arranged with decreasing lengths
 \begin{equation*}%
2\varepsilon \geq \ell(\beta_1) \geq \ell(\beta_2) \geq \dots \geq \ell(\beta_h).
 \end{equation*}%
The list may be void, but we shall assume that $n+h > 0$ to contrast this section from the preceding ones. We make use of the \emph{collar theorems} as e.g. Theorems 4.1.1, 4.1.6 and 4.4.6 in \cite{BuserBook} of which we recall the following.

For the cusps we have the \emph{cusp neighborhoods} $\mathscr{P}_i$, $i=1,\dots,n$, filled by the horocycles of lengths $< 2$. Each $\beta_k$ is simple and has a \emph{collar neighborhood} $\C_k$, $k=1,\dots,h$, filled by the points at distance $< w(\beta_k)$ from $\beta_k$, where $\sinh(w(\beta_k)) \sinh(\frac{1}{2}\ell(\beta_k)) = 1$. The two boundary curves of $\C_k$ are parallel curves to $\beta_k$ (i.e. all points have the same distance o $\beta_k$) and their lengths satisfy $\ell(\beta_k) \cdot \cosh(w(\beta_k)) > 2$. Topologically the $\mathscr{P}_i$ are punctured discs, the $\C_k$ are annuli and all these neighborhoods are pairwise disjoint. Finally we note that
 \begin{equation}\label{eq:hleq3gn}%
h \leq 3g -3 + n.
 \end{equation}%

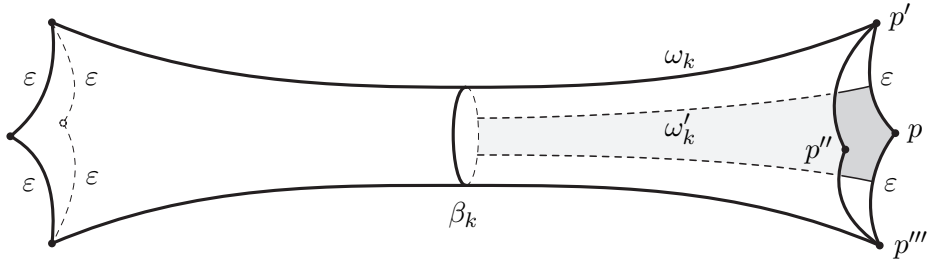
\begin{figure}[ht]
\leavevmode \SetLabels
(.51*.13) $\beta_k$\\
\L(0.883*.96) $p'$\\
\L(0.809*.40) $p''$\\
\L(0.887*.0) $p'''$\\
\L(0.899*.49) $p$\\
(0.70*.80) $\omega_k$\\
(0.70*.495) $\omega'_k$\\
(0.881*.71) $\varepsilon$\\
(0.884*.26) $\varepsilon$\\
(0.13*.26) $\varepsilon$\\
(0.131*.71) $\varepsilon$\\
(0.1855*.29) $\varepsilon$\\
(0.185*.71) $\varepsilon$\\
\endSetLabels
\begin{center}
\AffixLabels{\centerline{\epsfig{file =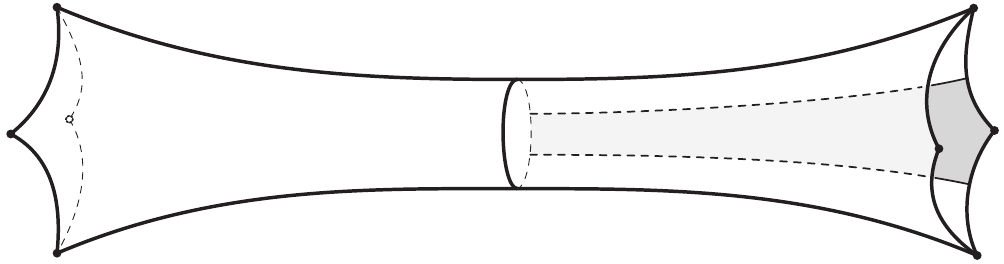,width=12.0cm,angle=0}}}
\vspace{-24pt}
\end{center}
\caption{\label{fig:ThinHandle}Reduced collar $\C'_k$ with part of the Voronoi cell at vertex $p$}\end{figure}

For $k = 1, \dots, h$ we now choose $\omega_k < w(\beta_k)$ in such a way that on either side of $\beta_k$ the parallel curve at distance $\omega_k$ admits a quadruple of points at consecutive distances $\varepsilon$ as the points $p, p', p'', p'''$ shown in Figure~\ref{fig:ThinHandle}. The value of $\omega_k$ is given by%
 \begin{equation*}%
\sinh(\sfrac{ \varepsilon }{2}) = \sinh(\frac{1}{8}\ell(\beta_k)) \cosh(\omega_k), 
 \end{equation*}%
which is {\bf formula \cite [2.3.1(v)]{BuserBook}} applied to the trirectangle with acute angle at $p'$ and consecutive sides of lengths $\frac{\varepsilon}{2}$, $\omega_k$, $\frac{1}{8} \ell(\beta_k)$, $\omega'_k$, where $\omega'_k$ is the distance between $\beta_k$ and the segment from $p$ to $p'$. From the formulas for $w(\beta_k)$ and $\omega_k$ we deduce by elementary computation that indeed $\omega_k < w(\beta_k)$. More accurately (using a numerical check for the upper bound),%
 \begin{equation}\label{eq:lowboundwk}%
 \arccosh(2) < \omega_k < w(\beta_k) - \frac{1}{3}.
 \end{equation}%

The same formula also yields the injectivity radius $r_p$ at $p$ (and $p'$, etc.). Indeed, the shortest geodesic loop at $p$ together with $\beta_k$ forms two identical trirectangles with consecutive sides $\omega_k, \frac{1}{2}\ell(\beta_k), \cdot \,, r_p$ and we have $\sinh(r_p) = \sinh(\frac{1}{2}\ell(\beta_k)) \cosh(\omega_k)$. Bringing the two formulas together we obtain
 \begin{equation*}%
\sinh(r_p) = \frac{\sinh(\frac{1}{2}\ell(\beta_k))}{\sinh(\frac{1}{8}\ell(\beta_k))}\sinh\big(\sfrac{ \varepsilon }{2}\big) > 4 \sinh\big(\sfrac{ \varepsilon }{2}\big).
 \end{equation*}%
Since $\varepsilon \leq \frac{1}{3}$ we have $r_p > 1.8 \varepsilon$. 

The geodesic segments of length $\varepsilon$ that connect the points $p,p',p'',p''',p$ form a simple curve homotopic to $\beta_k$ and there is another such curve on the other side of $\beta_k$. We choose it in such a way that the subset $\C'_k$ of $\C_k$ that lies between these two curves is symmetric with respect to $\beta_k$ as shown in Figure~\ref{fig:ThinHandle}. We call $\C'_k$ the \emph{reduced collar}.

In a similar manner we define the reduced cusp neighborhoods $\mathscr{P}'_i \subset \mathscr{P}_i$, $i=1,\dots,n$. For each $\mathscr{P}_i$ the vertices $p, p', p'', p'''$ on the boundary of $\mathscr{P}'_i $ lie on a horocycle. Since cusps may be viewed as limits of half collars we may apply the previous estimates taking the limit for $\ell(\beta_k) \to 0$. In particular the injectivity radius at the points $p$, $p'$, etc. has again the lower bound $r_p > 1.8 \varepsilon$. We call the union of the reduced cusps and collars the thin part of $S$ and the complement
 \begin{equation*}%
S' = S \setminus \big( \C'_1 \cup \dots \cup \C'_h \cup \mathscr{P}'_1 \cup \dots \cup \mathscr{P}'_n\big)
 \end{equation*}%
the thick part. By the collar theorems and the lower bound on $r_p$ the injectivity radius at any point in $S'$ is larger than $\varepsilon$.

The main result is the following. 

\begin{theorem}\label{thm:lq2} Let $S$ be as described and set 
 \begin{equation*}%
\sigma = \ell(\beta_1)\cdots \ell(\beta_h)
 \end{equation*}%
if $h\geq 1$ and $\sigma = 1$ if $h=0$. Then for any $\rho \leq \varepsilon$ and any disk $B_{\rho}$ of radius $\rho$ in the thick part of $S$ there exists a point $p \in B_{\rho}$ such that
\begin{equation*}
\dist(p,\BS) \geq \rho^2 \sigma^2 e^{-M(2g-2+n)},
\end{equation*}
where $M$ is a constant that depends only on $\varepsilon$.
\end{theorem}

The proof will be finalized in section~\ref{sec:finaliz2}, where we shall get the explicit bound $M = \frac{97}{\varepsilon^2} \log(\frac{134}{\varepsilon})$.

\subsection{An $\varepsilon$-net in the thick part and Voronoi cells in $S$}\label{sec:epsnetthick}

We construct an $\varepsilon$-net in $S'$ beginning with the above vertices $p$, $p', \dots$, on the boundaries of the reduced collars and cusp neighborhoods. We call these points the ``special points''. By the estimates in the preceding section and because collars and cusps are pairwise disjoint the special points have pairwise distances $\geq \varepsilon$. We now complete them into an $\varepsilon$-net on $S'$ by successively adding additional ``ordinary points'' on $S'$ at pairwise distances $\geq \varepsilon$ until this is no longer possible. For the properties of the resulting Voronoi cells we first prove the following. 
 \begin{lemma}\label{lem:specialpoints}%
For $x \in S \setminus S'$ the distance to the special points of the $\varepsilon$-net is smaller than the distance to the ordinary points.
 \end{lemma}%
 \begin{proof}%
Let $q$ be an ordinary point of the $\varepsilon$-net. The shortest connection from $x$ to $q$ intersects one of the boundary segments of length $\varepsilon$ in some point $y$. One of the endpoints of the segment, say $p$, satisfies $\dist(y,p) \leq \sfrac{ \varepsilon }{2}$. We now have $\dist(x,p) \leq \dist(x,y) + \dist(y,p) \leq \dist(x,y) + \sfrac{ \varepsilon }{2}$, where at least one of the inequalities is strict. On the other hand $\dist(x,q) = \dist(x,y) + \dist(y,q) \geq \dist(x,y) + \dist(p,q) - \sfrac{ \varepsilon }{2} \geq \dist(x,y)+ \sfrac{ \varepsilon }{2}$. Altogether $\dist(x,p) < \dist(x,q)$ which proves the claim.
 \end{proof}%

Our $\varepsilon$-net will now be considered as a distribution of points on $S$ and defines a decomposition of $S$ into Voronoi cells. We shall call ``special Voronoi cells" those around special points and ``ordinary Voronoi cells'' those around ordinary points. The cell decomposition consisting of all special and all ordinary Voronoi cells is again denoted by $\V$. The following properties take over from those in Section~\ref{sec:Voronoi}: 
 \begin{lemma}\label{lem:voronoisides}%
Any cell of $\V$ contains an embedded disk of radius $\sfrac{ \varepsilon }{2}$ and has at most 12 sides.
 \end{lemma}%
 \begin{proof}%
Lemma~\ref{lem:specialpoints} implies that the ordinary Voronoi cells are contained in $S'$ and have the same properties as those in Section~\ref{sec:Voronoi}. For the special cells we argue as follows. Consider, for instance, the special point $p$ on the boundary of the reduced collar $\C'_k$ as shown in Figure~\ref{fig:ThinHandle}. It follows from Lemma~\ref{lem:specialpoints} that the part in $\C'_k$ of the Voronoi cell at $p$ is the shaded polygon shown in the figure with two sides of length $\sfrac{ \varepsilon }{2}$ on the segments adjacent to $p$, then the two perpendiculars from these segments to $\beta_k$ and finally a side of length $\frac{1}{4} \ell(\beta_k)$ on $\beta_k$. The domain may be decomposed into two identical trirectangles with acute angle, say $\varphi$, at $p$. {\bf Formula \cite[2.3.1(iii)]{BuserBook}} yields $\cosh(\frac{1}{8}\ell(\beta_k))=\cosh(\sfrac{ \varepsilon }{2}) \sin \varphi$. Given that $\varepsilon \leq \frac{1}{3}$ and $\ell(\beta_k) \leq 2\varepsilon$ we deduce from it that the obtuse angle of the shaded domain at $p$ is $2 \varphi \geq 0.9 \pi$. Proceeding as in the proof of Lemma~\ref{lem:vornum} we now conclude that the Voronoi cell at $p$ has at most 9 sides. (From a combinatorial point of view, however, we count an additional degenerate side on $\beta_k$ since at the endpoints of the sides on $\beta_k$ there are four meeting cells.) Finally, with a glance at Figure~\ref{fig:ThinHandle} we see that the cell contains an embedded disk of radius $\sfrac{ \varepsilon }{2}$ centered at $p$.
 \end{proof}

When $p$ is on the boundary of a reduced cusp neighborhood then the result is the same except that the shaded domain is degenerate and one side is replaced by a point at infinity. Summing up we have the same conclusion as in Corollary~\ref{cor:edges} using that the area of $S$ is now $2\pi(2g-2+n)$
\begin{corollary}\label{cor:edges2}
The number of sides of $\V$ is bounded above by $ \left(\frac{97}{\varepsilon^2}-10\right)(g-1+\frac{n}{2})$.

\end{corollary}

\subsection{Traversing and terminal arcs}\label{sec:travend}

We now proceed as in the previous section to construct our model strands. As before, we consider an embedded disk $D$ of radius $\rho>0$ and a geodesic segment $c$ in $D$ belonging to the Birman-Series set. This time, however, $D$ is contained in the thick part $S'$.
 We take the continuation $\bar{c}$ of $c$ in both directions of length $R-2\varepsilon$. The endpoints lie in Voronoi cells. The only difference with the previous sections is the way in which we deal with the special Voronoi cells.
The intersections of $\bar{c}$ with the collars and cusps may be of two kinds: \emph{traversing arcs} and \emph{terminal arcs}. The precise definition and properties are as follows.


\emph{Traversing arcs.}\; Figure~\ref{fig:TraversingArc} shows half of a traversing arc $\mathcal{T}$ lifted to the universal covering of $S$. The horizontal line $\tilde{\beta}_k$ on the top is a lift of the geodesic $\beta_k$, the lower end is a lift of one of the boundary curves of $\C'_k$ consisting of geodesic segments of length $\varepsilon$. The dotted line is tangent to these and is a curve parallel to $\tilde{\beta}_k$ at distance $\omega'_k$. The vertices have distance $\omega_k$ to $\tilde{\beta}_k$. The meanings of $\omega_k$ and $\omega'_k$ are the same as in Section~\ref{sec:thickthin} (Figure~\ref{fig:ThinHandle}).

The projection (under the universal covering map) on $S$ of the strip between the dotted lines is the \emph{dotted collar} $\C''_k \subset \C'_k \subset \C_k$ defined as
 \begin{equation}\label{eq:dotted reduced}%
\C''_k = \{ x \in S \mid \dist(x,\beta_k) \leq \omega'_k \}.
 \end{equation}%

The boundary curves of $\C''_k$ are tangent to the boundary curves of $\C'_k$. In a similar way we define (for later use) in each cusp neighborhood $\mathscr{P}_i$ the \emph{dotted} cusp neighborhood $\mathscr{P}''_i \subset \mathscr{P}'_i \subset \mathscr{P}_i$ whose boundary curve is a horocycle tangent to the boundary curve of $\mathscr{P}'_i$.

Note that the traversing arcs begin and end on the dotted lines, i.e., any traversing arc $\mathcal{T}$ is a connected component, for some $k = 1, \dots, h$, of $\bar{c} \cap \C''_k$ that has its endpoints on the two opposite boundary curves of $\C''_k$.

The arc labelled $\lambda$ in Figure\ref{fig:TraversingArc} is half of the lift of $\mathcal{T}$ going from the dotted line to $\tilde{\beta}_k$. The label also denotes the length of this arc. Projecting $\lambda$ orthogonally to $\tilde{\beta}_k$ %
\begin{figure}[bt]
\leavevmode \SetLabels
(0.24*1.03) $\tilde{\beta}_k$\\
\L(0.162*.55) $\omega_k$\\
\L(0.225*.65) $\omega'_k$\\
\L(0.295*.65) $\omega'_k$\\
(0.275*.0) $\sfrac{\varepsilon}{2}$\\
(0.355*.0) $\sfrac{\varepsilon}{2}$\\
(0.62*.0) $\varepsilon$\\
(0.74*.0) $\varepsilon$\\
(0.45*.61) $\lambda$\\
(0.45*1.03) $\hat{\lambda}$\\
(0.318*-0.07) $\tilde{p}$\\
(0.345*.33) $\varepsilon$\\
(0.304*.0945) $A$\\
(0.34*.18) $\mathcal{U}$\\
\endSetLabels
\begin{center}
\AffixLabels{\centerline{\epsfig{file =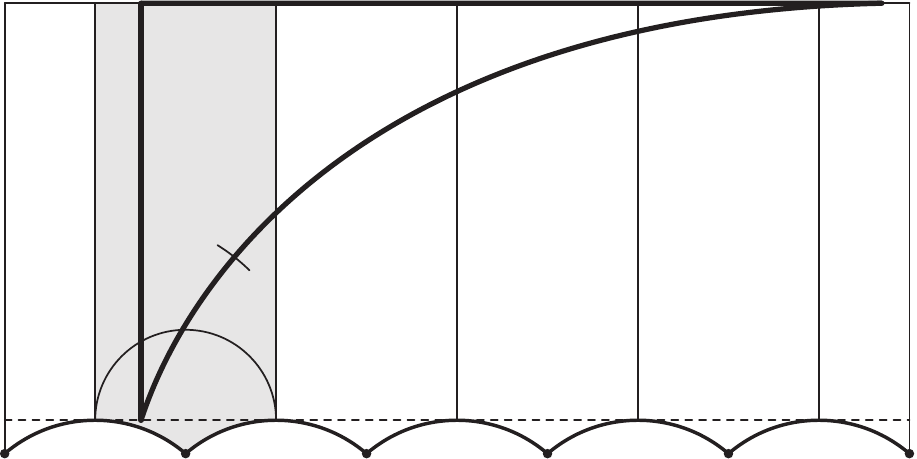,angle=0}}}
\vspace{-24pt}
\end{center}
\caption{\label{fig:TraversingArc}Traversing arc depicted in the universal covering}\end{figure}
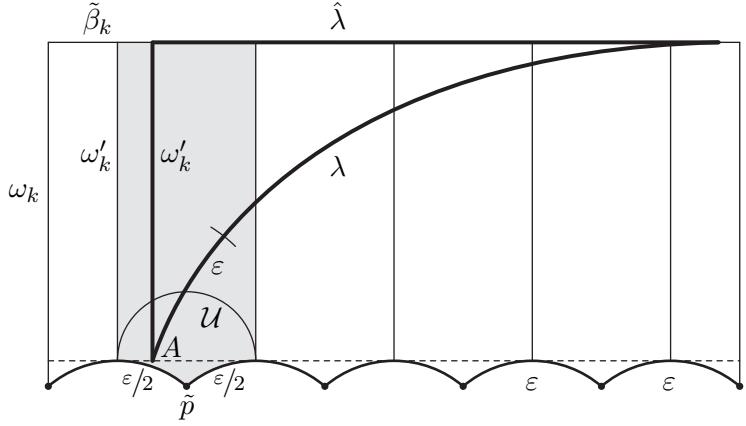
we obtain the leg $\hat{\lambda}$ of a right-angled geodesic triangle with hypothenuse $\lambda$ whose other leg is $\omega'_k$. We now compare $\lambda$ with $\hat{\lambda}$.

{\bf Formula \cite[2.2.2(i)]{BuserBook}} applied to the aforementioned right angled triangle yields $\cosh(\lambda) = \cosh(\hat{\lambda})\cosh(\omega'_k)$, and by {\bf formula \cite[2.3.1(iv)]{BuserBook}} applied to the trirectangle with sides $\frac{1}{2}\varepsilon, \omega'_k, \frac{1}{8}\ell(\beta_k), \omega_k$ we have $\cosh(\omega'_k)=\tanh(\sfrac{\varepsilon}{2})\coth(\frac{1}{8}\ell(\beta_k))$. Bringing this together and recalling that $\ell(\beta_k) \leq 2 \varepsilon$, $\varepsilon \leq 1/3$ we get $\cosh(\lambda) \geq \cosh(\hat{\lambda})\tanh(\sfrac{\varepsilon}{2})\coth(\sfrac{\varepsilon}{4}) \geq 1.98 \cosh(\hat{\lambda})$ and then, 
by an elementary estimate,
 \begin{equation}\label{eq:lamlam}%
\lambda - \hat{\lambda} > 2/3.
 \end{equation}%
The endpoint $A$ of $\lambda$ on the dotted line lies in a disk $\mathcal{U}$ of radius $\sfrac{\varepsilon}{2}$ around one of the vertices $\tilde{p}$. By the triangle inequality this implies, in turn, that \emph{$\mathcal{U}$ is contained in the disk of radius $\varepsilon$ around $A$}. 

\emph{Terminal arcs.}\; These are the connected components of the intersections of $\bar{c}$ with the dotted collars and cusp neighborhoods that have one endpoint on the boundary while the other lies in the interior. There are at most two such components. 

Since the infinite geodesic extension of $\bar{c}$ is a simple curve it follows that terminal arcs in a cusp neighborhood are ``vertical'' that is, orthogonal to the horocycles. In the collars the situation is different.

\begin{figure}[ht]
\leavevmode \SetLabels
(0.18*0.50) $\tilde{\beta}_k$\\
(0.493*0.72) $\omega'_k$\\
(0.215*0.30) $\omega'_k$\\
(0.377*0.71) $\sfrac{R'}{2}$\\
(0.33*0.30) $\sfrac{R'}{2}$\\
(0.43*0.53) $t$\\
(0.268*0.53) $\sigma_k$\\
(0.50*0.315) $R'$\\
(0.265*-0.055) $\tilde{p}$\\
(0.225*-0.022) $\tilde{q}$\\
(0.605*0.71) $\tilde{r}$\\
(0.24*1.03) $B$\\
(0.47*1.03) $C$\\
(0.627*0.395) $D$\\
(0.265*0.93) $\varepsilon'$\\
(0.325*0.93) $\varepsilon'$\\
(0.555*0.86) $\varepsilon$\\
(0.618*0.635) $\varepsilon$\\
(0.25*0.16) $\phi$\\
(0.455*0.85) $\phi$\\
(0.275*0.115) $\theta$\\
(0.785*0.455) $\eta$\\
\endSetLabels
\begin{center}
\AffixLabels{\centerline{\epsfig{file =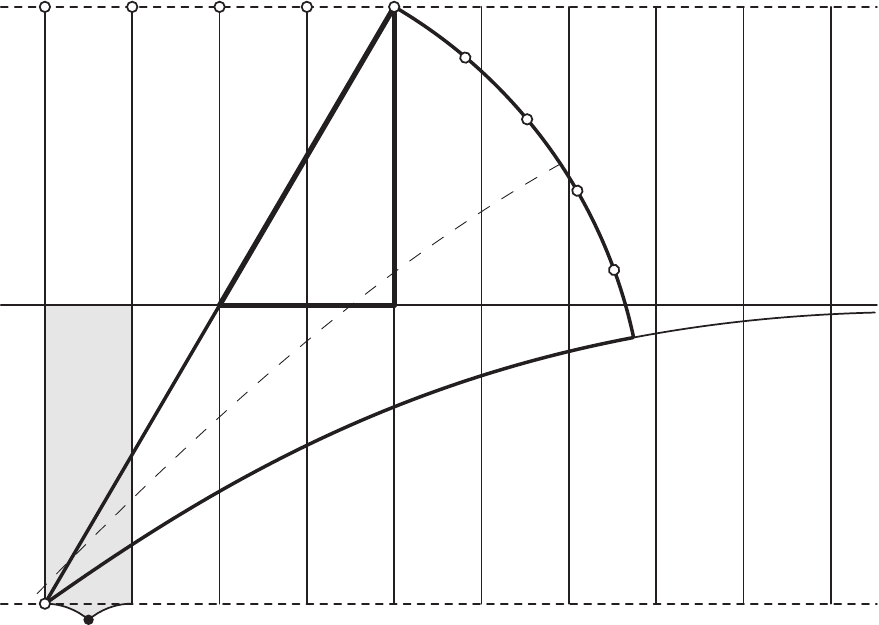,angle=0}}}
\vspace{-20pt}
\end{center}
\caption{\label{fig:EndArc}Endpoints for model strands in a reduced collar lifted to the universal covering. The dashed line is the lift of an extended terminal arc.}\end{figure}

To deal with the terminal arcs in $\C_k$ ($k = 1, \dots, h$) we supplement the $1$-skeleton of $\V$ by a set of \emph{{terminal segments}} since the $\varepsilon$-net has no vertices in $\C''_k$. For the description we use again the universal covering. Figure~\ref{fig:EndArc} shows the part of the lift of $\C_k$ that lies between the dotted lines. The horizontal line in the middle is a lift $\tilde{\beta}_k$ of $\beta_k$. The vertical lines are geodesic arcs orthogonal to $\tilde{\beta}_k$ at successive distances 
 \begin{equation*}%
\sigma_k = \frac{1}{4}\ell(\beta_k).
 \end{equation*}%

The shaded domain with vertices $\tilde{p}$, $\tilde{q}$ and sides of lengths $\sfrac{\varepsilon}{2}, \omega'_k, \sigma_k, \omega'_k, \sfrac{\varepsilon}{2}$ has the same meaning as in Figure~\ref{fig:TraversingArc} and is a lift of the shaded domain shown in Figure~\ref{fig:ThinHandle}. 

The terminal segments will be issued at the midpoints of the boundary segments of lengths $\varepsilon$. We describe these arcs in the case of the midpoint point $q$ at the boundary of $\C'_k$ whose lift $\tilde{q}$ in the universal covering is shown in Figure~\ref{fig:EndArc}. We shall assume that $R \geq 2 \omega'_k$ and add the simple modification for $R < 2 \omega'_k$ at the end.

The midpoints of the boundary segments on the dotted line on the top -- we shall call them ``white points'' -- have successive distances of some value $\varepsilon'$ slightly smaller than $\varepsilon$. We let $B$ be the white point on the top opposite to $\tilde{q}$ and then $C$ the white point to the right of $B$ whose distance $R'$ to $\tilde{q}$ is as close to $R$ as possible. Then $R -\varepsilon' \leq R' \leq R + \varepsilon'$. For the number $z_1$ of white points from $B$ to $C$, not counting $B$, we have the following, where $t$ is the orthogonal projection of the segment of length $\sfrac{R'}{2}$ onto $\tilde{\beta}_k$ (Figure~\ref{fig:EndArc}),
 \begin{equation*}%
z_1 = \frac{2t}{\sigma_k} \leq \frac{R'}{\sigma_k} \leq \frac{R+\varepsilon}{\sigma_k}
 \end{equation*}%
Some of the comparison arcs will go from $\tilde{q}$ to these points, but we need further ones. To this end we look at the geodesic ray $\eta$ emanating from $\tilde{q}$ asymptotic to $\tilde{\beta_k}$ as shown in Figure~\ref{fig:EndArc}. Together with $\tilde{\beta}_k$ and the vertical line at $\tilde{q}$ it forms an ideal right-angled triangle with finite leg $\omega'_k$ and acute angle at $\tilde{q}$ which we write as a sum $\phi + \theta$, $\phi$ being the angle at $\tilde{q}$ of the right-angled triangle with hypothenuse $\sfrac{R'}{2}$ and legs $\omega'_k$, $t$. 
By {\bf formula \cite[2.2.2(vi)]{BuserBook}}
 \begin{equation*}%
\cos(\phi) = \frac{\tanh(\omega'_k)}{\tanh(\sfrac{R'}{2})}, \quad \cos(\phi + \theta) = \tanh(\omega'_k).
 \end{equation*}%
Out of these relations we get, by elementary transformations,%
 \begin{equation*}%
\sin(\theta) = \frac{\tanh(\omega'_k)}{\tanh(\sfrac{R'}{2})}\cdot \frac{1}{\cosh(\omega'_k)}\left\{1 - \sqrt{1- \frac{\cosh(\omega'_k)^2}{\cosh(\sfrac{R'}{2})^2}}\right\}.
 \end{equation*}%
Using that for $0 \leq x \leq 1$ we have $1-\sqrt{1-x} \leq x$ and that $\theta < \sfrac{\pi}{2} \sin(\theta)$ we further obtain
%
 \begin{equation*}%
\theta < \frac{\pi\sinh(\omega'_k)}{\sinh(R')}.
 \end{equation*}%

We now introduce additional white points on the arc that goes from $C$ to $D \in \eta$ on the circle with radius $R'$ and center $\tilde{q}$ as shown in Figure~\ref{fig:EndArc}. We position them in such a way that the successive arcs between them have lengths $\varepsilon$. Let $z_2$ be the number of these points, not counting $C$. Then $z_2 \leq \frac{\theta}{\varepsilon} \sinh(R')<\frac{\pi}{\varepsilon} \sinh(\omega_k')$. Using that $\cosh(\omega'_k) = \tanh(\sfrac{\varepsilon}{2})\coth(\sfrac{\sigma_k}{2}) \leq\frac{\varepsilon}{\sigma_k}$, {\bf formula \cite[2.3.1(iv)]{BuserBook}} applied to the trirectangle with sides $\sfrac{\varepsilon}{2}, \omega'_k, \sfrac{\sigma_k}{2}, \omega_k$) we further get
 \begin{equation*}%
z_2 < \frac{\pi}{\sigma_k}.
 \end{equation*}%
Now $z_1+z_2$ is the number of white points to the right of $B$ (Figure~\ref{fig:EndArc}) and there is the same number of similar points on the left. Drawing the connecting geodesic arcs from $\tilde{q}$ to these points (including the arc from $\tilde{q}$ to $B$) and projecting them from the universal covering to $S$ we get the \emph{terminal segments at} $q$.

So far we have assumed that $R \geq2 \omega'_k$. If $R < 2 \omega'_k$, then $z_1 = 0$ and the number of white points on the circular arc of radius $R$ is smaller than or equal to the number of such points we would get on the arc of radius $2 \omega'_k$. For the latter we have already found the bound $\frac{\pi}{\sigma_k}$. Hence again $z_2 < \frac{\pi}{\sigma_k}$.

Summing up we have the following.
 \begin{lemma}\label{lem:ModelEnd}%
At any white point $q$ on the boundary of $\C''_k$ there are at most $\frac{2}{\sigma_k}(R+4)$ terminal segments.
 \end{lemma}%

Finally, if $q$ is a white point on the boundary of a reduced cusp neighborhood $\mathscr{P}'_i$ then we need only one terminal segment, namely the geodesic arc of length $R$ at $q$ that is orthogonal to the horocycles.

 \mbox{}

%

\subsection{Constructing the model strands}\label{sec:modelstrands} 
We begin by assigning to each collar $\C''_k$ a \emph{winding number} $\tau_k$ with respect to $\bar{c}$ as follows. If $\C''_k$ contains no traversing arc we set $\tau_k = 0$. Otherwise we let $\mathcal{T}_k$ be the longest traversing arc in $\C''_k$, project it orthogonally to a parametrized curve $\hat{\mathcal{T}}_k$ on $\beta_k$ (that may go around $\beta_k$ many times) and set
\begin{equation}\label{eq:windingnumber}%
\tau_k = s_k \left[ \frac{\ell(\hat{\mathcal{T}}_k)}{\ell(\beta_k)} \right],
\end{equation}%
where $[x]$ for $x \in \R$ denotes the largest integer $\leq x$ and $s_k \in \{-1,1\}$ is the orientation, i.e., $s_k = 1$ if $\hat{\mathcal{T}}_k$
winds around $\beta_k$ in the positive sense (with respect to a fixed orientation of $\beta_k$) and $s_k = -1$ otherwise. At some later point we shall simultaneously unwind all traversing arcs in $\C''_k$ by applying a Dehn twist $D_k$ of order $\tau_k$ along $\beta_k$ in the ``unwinding direction'': if $\mathcal{T}$ traverses $\C''_k$ from, say $A$ to $A'$, then $D_k(\mathcal{T})$ is the geodesic arc from $A$ to $A'$ in $\C''_k$ that is homotopic (with fixed endpoints) to the curve $\mathcal{T}'$ that goes along $\mathcal{T}$ from $A$ to $\beta_k$, then $\vert\tau_k\vert$ times around $\beta_k$ in the opposite direction of $\hat{\mathcal{T}}_k$ and after that along $\mathcal{T}$ to $A'$. It follows from this construction that the orthogonal projection of $D_k(\mathcal{T}_k)$ on $\beta_k$ has length $< \ell(\beta_k)$. Since the traversing arcs are pairwise disjoint we conclude that all of them have this property, i.e.,\ \emph{any $D_k(\mathcal{T})$ winds less than once around $\C''_k$}.

The model strand for $\bar{c}$ is constructed in a similar fashion as in Section~\ref{sec:locqua}, though with a modification for the traversing arcs and the terminal arcs. We proceed in three steps. First we split $\bar{c}$ into a product
 \begin{equation*}%
\bar{c} = c_0 c_1 c_2 c_3 \cdots c_{2J+1}c_{2J+2}
 \end{equation*}%
($J\geq0$), where $c_0$ and $c_{2J+2}$ are either terminal arcs (Section~\ref{sec:travend}) or point curves depending on whether or not $\bar{c}$ begins, respectively ends with a terminal arc; the parts $c_1, c_3, \dots, c_{2J+1}$ are outside the dotted collars and dotted cusp neighborhoods; 
the parts $c_{2j}$, for $j=1,\dots,J$, comprise all the traversing arcs, each of them traversing some dotted collar $\C''_{k_j}$. Here we are using that a complete simple geodesic cannot enter and leave a collar on the same side (Lemma~\ref{lem:collar}) and that, similarly, it cannot enter and leave a cusp neighborhood.
%
To deal with the windings separately we first ``unwind'' $\bar{c}$\, and also leave out $c_0$ and $c_{2J+2}$ setting
 \begin{equation*}%
\breve{c} = c_1 \breve{c}_2 c_3\breve{c}_4c_5 \cdots c_{2J-1} \breve{c}_{2J} c_{2J+1},
 \end{equation*}%
where $\breve{c}_{2j} = D_{k_j}(c_{2j})$, $j=1, \dots, J$. Now $\breve{c}$ traverses a succession of Voronoi cells and we associate to it a polygonal curve $P(\breve{c})$ by the same procedure as in Section~\ref{sec.concom} with a minor adaptation: if $c_0$ is a terminal arc, say $c_0 \subset \C''_{\ell}$, then $P(\breve{c})$ has its initial point on the boundary of $\C''_{\ell}$ (as though $\C''_{\ell}$ was the 0-th Voronoi cell that $\breve{c}$ goes along). This initial point is then one of the white points. Similarly, if $c_{2J+2}$ is a terminal arc in some $\C''_k$, then the endpoint of $P(\breve{c})$ is a white point on the boundary of $\C''_k$.

The curves $\breve{c}$ and $P(\breve{c})$ have nearby endpoints: there is a geodesic arc $\breve{u}$ of length $\leq 2 \varepsilon$ leading from the initial point of $\breve{c}$ to the initial point of $P(\breve{c})$ and a geodesic arc $\breve{v}$ of length $\leq 2 \varepsilon$ leading from the endpoint of $P(\breve{c})$ to the endpoint of $\breve{c}$. The paths $\breve{c}$ and $\breve{u}P(\breve{c})\breve{v}$ are homotopic.

In the next step we apply to $P(\breve{c})$ the reversed Dehn twists $D_k^{-1}$, $k=1,\dots,h$. The resulting curve $\overfrown{P}(\breve{c}) = D_1^{-1} \cdots D_h^{-1}(P(\breve{c}))$ has the same endpoints as $P(\breve{c})$ and $\breve{u}\overfrown{P}(\breve{c})\breve{v}$ is homotopic to the curve $c_1c_2c_3c_4\cdots c_{2J+1}$.

In the final step we add the possible terminal segments. We describe this for $c_{2J+2}$, the procedure for $c_0$ being the same. If $c_{2J+2}$ is a point curve we set $\breve{c}_{2J+2} = c_{2J+2}$. Now assume that $c_{2J+2}$ is a terminal arc in some dotted collar $\C''_k$. Then the initial point of $c_{2J+2}$ lies within distance $\leq \sfrac{\varepsilon}{2}$ of the endpoint, say $q$, of $\overfrown{P}(\breve{c})$ on the boundary of $\C''_k$ and we first extend $c_{2J+2}$ to a longer geodesic arc $c'_{2J+2}$ in such a way that if we lift it to the universal covering, then the other endpoint lies within distance $\leq \sfrac{\varepsilon}{2}$ of some white point, say $\tilde{r}$, on the arc $BCD$ as depicted in Figure~\ref{fig:EndArc} (or its symmetric image across $\tilde{q}B$), where the dotted line is $c'_{2J+2}$. The geodesic from $\tilde{q}$ to $\tilde{r}$ is the lift of a terminal segment and we let 
$\breve{c}_{2J+2}$ be this terminal segment. There is a connecting arc $v$ of length $\leq \sfrac{\varepsilon}{2}$ from the endpoint of $\breve{c}_{2J+2}$ to the endpoint of
$c'_{2J+2}$, and the curves $(\breve{v})^{-1} \breve{c}_{2J+2} v$ and $c'_{2J+2}$ are homotopic. Finally, if $c_{2J+2}$ is a terminal arc in some dotted cusp neighborhood $\mathscr{P}''_i$, then we proceed similarly except that we take the extension $c'_{2J+2}$ to be of length $R$. In this latter case $\breve{c}_{2J+2}$ is a geodesic arc orthogonal to the horocycles.

We now let $\bar{c}'$ be the extended geodesic arc $\bar{c}' = c_0' c_1 c_2 c_3 c_4 \cdots c_{2J+1} c'_{2J+2}$ (where $c'_0$ is defined in the same way as $c'_{2J+2}$) and define the model strand $m^c$ to be the geodesic arc in the homotopy class of $\breve{c}_0 \overfrown{P}(\breve{c})\breve{c}_{2J+2}$. Then there is a connecting arc $u$ of length $\leq 2 \varepsilon$ from the initial point of $\bar{c}'$ to the initial point of $m^c$, and the already described connecting arc $v$ of length $\leq 2 \varepsilon$ from the endpoint of $m^c$ to the endpoint of $\bar{c}'$. By the aforementioned homotopies the curves $u \,m^c v$ and $\bar{c}'$ are homotopic. In the disk $D$ we have therefore the same estimate for the distance between $c$ and $m_c$ as in \eqref{eq:tubular} and so we get the following result for our model strands $m^c$ and their components $m_c$ in $D$:

%
 \begin{lemma}\label{lem:distcmc}%
Any component $c$ of the Birman-Series set in $D$ lies in a tubular neighborhood of radius $w_R := 3e^{-R}$ of the model strand $m_c$.
 \end{lemma}%

\subsection{Estimating the number of model strands}\label{sec:estmodelstrands} 

 \begin{lemma}\label{lem:edgetwist}%
Let $\breve{L}$ be the number of path edges of $P(\breve{c})$ and $\tau_1, \dots \tau_h$ the winding numbers of $\bar{c}$ at $\beta_1, \dots, \beta_h$ as in \eqref{eq:windingnumber}, then

 \begin{equation*}%
\frac{\varepsilon}{24}\, \breve{L} + \sum_{k=1}^{h}\ell(\beta_k) \vert\tau_k\vert \leq \ell(\bar{c}) + 3 \varepsilon.
 \end{equation*}%
 \end{lemma}%
 \begin{proof}%
For $j=0,\dots,J$, any Voronoi cell (including the special ones) that is crossed by $c_{2j+1}$ has its center at distance $\leq \varepsilon$ from $c_{2j+1}$ (by the completeness of the $\varepsilon$ net in $S'$). By the same area argument as in the proof of Lemma~\ref{lem:comblength} it follows that $c_{2j+1}$ crosses at most $\frac{4}{\varepsilon}(\ell(c_{2j+1})+3\varepsilon)$ times some Voronoi cell. Each $\breve{c}_{2j}$, for $j=1, \dots, J$, crosses an additional number of at most 4 special Voronoi cells. This is so because $\breve{c}_{2j}$ winds at most once around the corresponding collar and its initial and final cells are on the accounts of $c_{2j-1}$ and $c_{2j+1}$, respectively. Altogether $\breve{c}$ crosses at most $\frac{4}{\varepsilon}\sum_{j=0}^J(\ell(c_{2j+1})+3\varepsilon) + 4 J$ times some Voronoi cell and this multiplied by 6 is an upper bound of $\breve{L}$:
 \begin{equation}\label{eq:ineqA}%
\breve{L} \leq \frac{24}{\varepsilon}\sum_{j=0}^J(\ell(c_{2j+1})+3\varepsilon) + 24 J.
 \end{equation}%
For any $k=1,\dots,h$, such that $\tau_k \neq 0$ there is the longest traversing arc $\mathcal{T}_k$ in $\C''_k$ which is one of the $c_{2j}$ say $\mathcal{T}_k = c_{2j(k)}$. By \eqref{eq:lamlam} its orthogonal projection $\hat{c}_{2j(k)}$ on $\beta_k$ has length satisfying $ \ell(\beta_k) \vert \tau_k \vert \leq \ell(\hat{c}_{2j(k)}) < \ell(c_{2j(k)}) -\sfrac{4}{3}\leq \ell(c_{2j(k)})-4 \varepsilon$. Using that any other traversing arc has length $\ell(c_{2j}) > 4 \epsilon$ 
we get
 \begin{equation}\label{eq:ineqB}%
\sum_{k=1}^h \ell(\beta_k)\vert \tau_k \vert \leq \sum_{j=1}^{J}(\ell(c_{2j})-4\varepsilon)
 \end{equation}%
With \eqref{eq:ineqA} and \eqref{eq:ineqB} the proof is complete.
\end{proof}%
%
%
In the next lemma the constant $G$ stems from Corollary~\ref{cor:edges2}. In the case $n=0$ it is the same as in \eqref{eq:defG}, Section~\ref{sec:finaliz}.
 \begin{lemma}\label{lem:numbermc}%
Set $G = \frac{97}{\varepsilon^2}(g-1+\frac{n}{2})$ and
 \begin{equation*}%
\mathcal{N}(R) = \frac{1}{10}\, \frac{m^G}{G^G}R^G \frac{1}{\ell(\beta_1)\cdots\ell(\beta_h)},
 \end{equation*}%
with $m = \frac{134}{\varepsilon}$. Then for given $R > G$ there are at most $\mathcal{N}(R)$ model strands in $D$.
 \end{lemma}%
 \begin{proof}%
We first count how many model strands arise from extensions $\bar{c}$ that have both endpoints outside the dotted collars. We shall say that these curves belong to the \emph{first category}. For any of them the model strand $m^c$ is uniquely determined by the following three sets of data: the winding numbers $\tau_1, \dots, \tau_h$, the numbers $n_1, \dots, n_V$ of path edges of $P(\breve{c})$ on each of the $V$ edges of the cell decomposition $\V$ (which we assume enumerated from $1$ to $V$ in some way) and the selection of the initial and end path edge. By Lemma~\ref{lem:edgetwist} we have, noting that $\breve{L} = n_1+\dots+n_V$, and $\ell(\bar{c}) \leq 2R - 4 \varepsilon$,
 \begin{equation}\label{eq:ineqn1nV}%
\frac{\varepsilon}{24}(n_1+\dots+n_V) + \sum_{k=1}^{h}\ell(\beta_k) \vert\tau_k\vert \leq 2R.
 \end{equation}%
The number of model strands in the present case is thus bounded above by $4\breve{L}^2 2^h$ times the number of strings $(n_1, \dots, n_V, \vert \tau_1\vert, \dots, \vert \tau_h \vert)$ satisfying inequality \eqref{eq:ineqn1nV}. By Corollary~\ref{cor:comblem} this number is bounded above by

 \begin{equation}\label{eq:numbstring}%
B := \frac{1}{(V+h)! \left(\frac{\varepsilon}{24}\right)^V \ell(\beta_1) \cdots \ell(\beta_h)}(2R + \gamma)^{V+h},
 \end{equation}%
where $\gamma= \frac{\varepsilon}{24} V+ \ell(\beta_1) + \dots + \ell(\beta_h)$. Hence, the upper bound $4\breve{L}^2 2^h B$ for the number of model strands with endpoints outside the dotted collars. 

We also note using \eqref{eq:hleq3gn} and Corollary~\ref{cor:edges2} that $h \leq 3(g-1+\frac{n}{2})$ and hence,
 \begin{equation}\label{eq:boundgamma}%
V+h\leq G-5, \qquad \gamma < \frac{1}{50} R.
 \end{equation}%
We extend our first category of curves by allowing, in addition, the endpoints to lie in dotted collars that contain traversing arcs. For this extended category we have the previous bound multiplied by the number of possible choices of the terminal segments that are attached to $\overfrown{P}(\breve{c})$. Any such segment must be homotopic with fixed endpoints to a curve that does not intersects $\overfrown{P}(\breve{c})$ except at the endpoints. One of the latter is the attachement to $\overfrown{P}(\breve{c})$, the other is among the four white points on the opposite boundary component of the corresponding dotted collar. Hence, at either end of $\overfrown{P}(\breve{c})$ there are at most 5 possible choices for the terminal segment. This yields the bound $100\breve{L}^2 2^h B$ for the number of curves in the extended first category.

The \emph{second category} consists of the cases where exactly one endpoint of $\bar{c}$ lies in a dotted collar that contains \emph{no} traversing arcs or where both endpoints lie in the same dotted collar that contains no traversing arcs. We shall count the arising model strands for the cases for which this collar is $\C''_k$ and then take the sum for $k=1,\dots,h$. Now for $\bar{c}$ with initial point in $\C''_k$ (and possibly the end point also but without traversing arcs in $\C''_k$) the attachment point for the terminal segment that is glued to the beginning of $\overfrown{P}(\breve{c})$ is uniquely determined by the sequence $(n_1,\dots,n_V)$, but the previous bound of 5 for the possible directions is now replaced by $\frac{8}{\ell(\beta_k)}(R+4)$ (Lemma~\ref{lem:ModelEnd}). At the same time, since \CG{$\tau_k = 0$}, inequality \eqref{eq:ineqn1nV} is now replaced by $\frac{\varepsilon}{24}(n_1+\dots+n_V) + \sum_{j=1,j\neq k}^{h}\ell(\beta_k) \vert\tau_k\vert \leq 2R$, and Corollary~\ref{cor:comblem} yields the bound
 \begin{equation*}
B' := \frac{\ell(\beta_k)}{(V+h-1)! \left(\frac{\varepsilon}{24}\right)^{V} \ell(\beta_1) \cdots \ell(\beta_{h})}(2R + \gamma)^{V+h-1}.
 \end{equation*}%
for the number of strings that satisfy it. If the second endpoint of $\bar{c}$ happens to lie in $\C''_k$ also, then both attachment points are determined by $(n_1, \dots, n_V)$ but for given direction of the terminal segment at the beginning there are at most 5 directions for the terminal segment at the end, owing to the fact that both terminal arcs of $\bar{c}$ in $\C''_k$ lie on the same simple geodesic. It follows that for $\C''_k$ there are at most $10\breve{L} 2^{h-1} B' \frac{8}{\ell(\beta_k)}(R+4) \leq 40 \breve{L} 2^h (V+h) B$ possible cases, and summing up for $k = 1, \dots, h$ we get the upper bound $40 h \breve{L} 2^h (V+h) B$ for the number of model strands arising from curves in category 2.

The \emph{third} and final category consists of the cases where the two endpoints of $\bar{c}$ lie in distinct dotted collars that both contain no traversing arcs. A similar argument as before shows that there are at most $8h^2 2^h (V+h)(V+h-1)B$ model strands arising from this last category.

Let now $N$ be the sum of the bounds for the three categories. By Lemma~\ref{lem:edgetwist} and since $\ell(\bar{c}) \leq 2R-4$ (Section~\ref{sec:travend}) we have $\breve{L} \leq \frac{48}{\varepsilon} R$ \eqref{eq:boundgamma}. Furthermore, by the hypothesis of the lemma, $G < R$ and therefore $V+h < R$. Hence, allowing rough estimates at this point,
 \begin{equation*}%
N \leq 100\left( \frac{48}{\varepsilon}\right)^2 R^2 \{1 + \frac{1}{10} h + \frac{1}{100}h^2 \} 2^h B \leq 100\left( \frac{48}{\varepsilon}\right)^2 R^2 3^h B. 
 \end{equation*}%
Applying to $B$ that $\gamma < \frac{1}{50}R$ \eqref{eq:boundgamma} and using that $x! \geq x^x e^{-x}$ for $x > 0$ we get %
 \begin{equation*}%
B \leq \frac{1}{\ell(\beta_1)\cdots \ell(\beta_h)} \left(\frac{24}{\varepsilon}\right)^V\left(2+\frac{1}{50}\right)^{V+h}\left(\frac{e R}{V+h}\right)^{V+h}.
 \end{equation*}%
Since the function $x \to (e R/x)^x$ is monotone increasing for $x \in [1, R]$ and, by \eqref{eq:boundgamma}, $V+h \leq G-2 < R$, the last factor has the bound $(eR/(V+h))^{V+h}\leq (eR/(G-2))^{G-2}$. The bound in the lemma now follows by elementary simplification using that $G>400$.
\end{proof}%
\subsection{Finalizing the proof}\label{sec:finaliz2} 
Theorem~\ref{thm:lq2} is now proved by the same argument as in the proof of Theorem~\ref{thm:lq1}: the given disc $B_{\rho}$ has area $\geq \pi \rho^2$; for any model strand $m_c$ in $B_{\rho}$ the part $m_c^{2w_R}$ of the tubular neighborhood of radius $2w_R$ that lies in the disc has area $< 9\, \rho\, w_R$ (see \eqref{eq:areamcwR}), where $w_R = 3e^{-R}$. By Lemma~\ref{lem:numbermc} there are at most $\mathcal{N}(R)$ such strands and by Lemma~\ref{lem:distcmc} the tubular neighborhoods of radius $w_R$ around them contain the Birman-Series set $\cap B_{\rho}$. We thus get a point $p$ in $B_{\rho}$ at distance $\geq w_R$ from it if we set $R$ to be a solution $\geq G$ to the equation 
 \begin{equation*}%
\frac{1}{\rho \, \sigma} \, \frac{m^G}{G^G} R^G = e^R.
 \end{equation*}%
This is \eqref{eq:theequation} with $\rho$ replaced by $\rho \, \sigma$. By Lemma~\ref{lem:solutionbound}
this solution satisfies
 \begin{equation}\label{eq:boundRzero}%
R \leq 2 \log(\frac{1}{\rho\, \sigma}) + 2G \log(m) 
 \end{equation}%
and, analogously to \eqref{eq.final bound}, we get the bound
$w_R \geq 3\rho^2\sigma^2 e^{-2G \log{m}}$ with $G$ and $m$ as in Lemma~\ref{lem:numbermc}. This completes the proof.\hfill \qedsymbol

\section{Systole independent bounds}\label{sec:sysindep}
From a geometric point of view collars around very small geodesics are similar to pairs of cusps. This suggests that there should also exist a version of Theorem~\ref{thm:lq1} with bounds that are independent of the systole. Here we show that this is indeed possible. However, the constants arising from our approach become extremely small.

The main idea is that if the width of a collar is sufficiently large with respect to $R$ then it cannot contain traversing arcs and, furthermore, the number of possible directions of terminal segments (c.f.\ Lemma~\ref{lem:ModelEnd}) at any white point on the boundary is just equal to 1 as in the case of a cusp.

In Figure~\ref{fig:WindingArc} we calculate how much larger than $R$ the width must be. The figure depicts again part of the universal covering of the surface $S$ and is in correspondence with the earlier Figure~\ref{fig:EndArc}. The distance from the white point $\tilde{q}$ to the lift $\tilde{\beta}_k$ of the small geodesic $\beta_k$ is equal to $R+H$ with $H$ to be determined and $\sigma_k = \sfrac{1}{4} \ell(\beta_k)$. The geodesic ray $\eta$ issued at $\tilde{q}$ is asymptotic to $\tilde{\beta}_k$ and forms an angle $\theta$ with the vertical geodesic from $\tilde{q}$ orthogonally to $\tilde{\beta}_k$. We now determine $H$ in such a way that $\eta$ contains the hypothenuse of a geodesic triangle $\tilde{q} CD$ with right angle at $C$ and small sides $R$ and $\varepsilon$. The latter is almost identical with the arc $CD$ from $C$ to $\eta$ on the cricle of radius $R$ centered at $\tilde{q}$. Speaking in terms of Figure~\ref{fig:EndArc} we have $B=C$ and there is only one white point on the curve $BCD$ (and its symmetric image across $\tilde{q}C$) and thus only one choice for the direction of a terminal segment at $q$ in $S$. The necessary distance $H$ is determined by the following triangle formulas ({\bf formula \cite[2.2.2(iv)]{BuserBook}})
 \begin{equation}\label{eq:determineH}%
\sinh(R) = \tanh(\varepsilon) \cot(\theta), \quad \sinh(R+H) = \cot(\theta).
 \end{equation}%
By the formula for the grey shaded trirectangle already used prior to \eqref{eq:lamlam} (with $\omega'_k = R+H$) we have $\cosh(R+H) = \tanh(\sfrac{\varepsilon}{2}) \coth(\sfrac{\sigma_k}{2}) $. For $H$ determined by \eqref{eq:determineH} the length of $\beta_k$, with negligible error, is equal to $16 \tanh(\sfrac{\varepsilon}{2})\tanh(\varepsilon) e^{-R}$. Thus, if we define%
 \begin{equation}\label{eq:betacondition}%
\mathcal{L}(R) = 4 \varepsilon^2 e^{-R}
 \end{equation}%
then for given $R$ any collar $\C''_k$ with $\ell(\beta_k) \leq \mathcal{L}(R)$ may be dealt with as though it was a pair of cusps.

\begin{figure}[t]
\leavevmode \SetLabels
(0.45*1.03) $\tilde{\beta}_k$\\
(0.244*.10) $\sfrac{\varepsilon}{2}$\\
(0.315*.788) $\varepsilon$\\
(0.29*.018) $\tilde{q}$\\
(0.265*.696) $C$\\
(0.353*.696) $D$\\
(0.265*.45) $R$\\
(0.265*.85) $H$\\
(0.60*.92) $\eta$\\
(0.335*.31) $\theta$\\
(0.244*1.03) $\sfrac{\sigma_k}{2}$\\
\endSetLabels
\begin{center}
\AffixLabels{\centerline{\epsfig{file =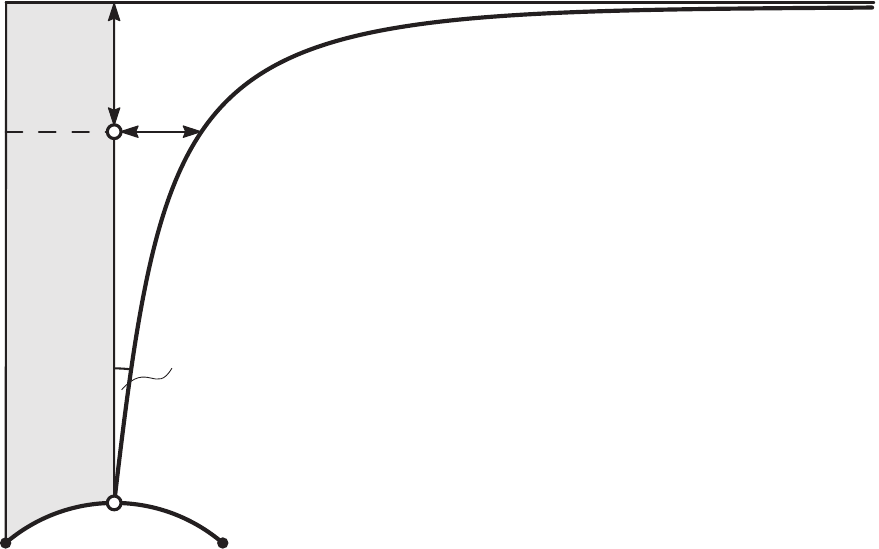,angle=0}}}
\vspace{-24pt}
\end{center}
\caption{\label{fig:WindingArc}Asymptotic geodesic depicted in the universal covering}\end{figure}
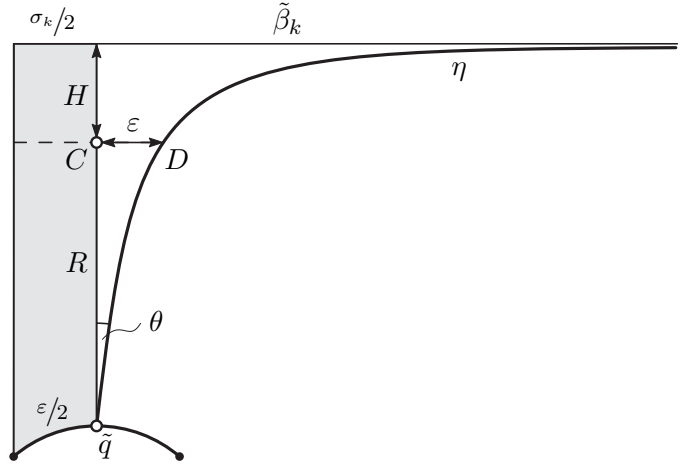
We now start an iteration beginning by setting
 \begin{equation*}%
\ell_0 = 1, \quad R_0 = 2\log(\frac{1}{\rho}) + 2G \log(m),
 \end{equation*}%
where $\rho$, $G$, $m$ are as in Theorem~\ref{thm:lq2} and Lemma~\ref{lem:numbermc}. The value for $R_0$ stems from \eqref{eq:boundRzero} and is the bound for $R$ in the proof of Theorem~\ref{thm:lq2} that holds if $h = 0$, i.e.,\ in the case where $S$ has no small geodesics. By what we have said above, this bound is also valid if all small geodesics on $S$ have lengths $\leq \mathcal{L}(R_0)$, and so under this weaker hypothesis we still have the lower bound $3 e^{-R_0}$ for the largest distance to the Birman-Series set in the disc $B_{\rho}$.

In the first iteration step we set $\ell_1 = \mathcal{L}(R_0)$ and set $R_1$ equal to the right hand side of \eqref{eq:boundRzero} for the special case $\sigma = \ell_1$. This is then the bound for $R$ in the proof of Theorem~\ref{thm:lq2} that holds if $h = 1$, and the same bound is valid in the more general case where all small geodesics different from $\beta_1$ have lengths $\leq \mathcal{L}(R_1)$. In this way we continue getting two sequences sequence $\ell_0, \ell_1, \ell_2, \dots$ and $R_0, R_1, R_2, \dots$, with the iteration scheme

\medskip

\hspace{2em}$\ell_k = \mathcal{L}(R_{k-1}) = 4 \varepsilon^2 e^{-R_{k-1}}$

\hspace{2em}$ R_k = 2 \log{\frac{1}{\rho}} + 2 \log\frac{1}{\ell_1} + \cdots + 2 \log\frac{1}{\ell_k} + 2G\, \log(m)$

\medskip
and at each step the result that some point in $B_{\rho}$ has distance $\geq w_{R_k} = 3 e^{-R_k}$ to the Birman-Series set provided that on $S$ all small geodesics different from $\beta_1, \dots, \beta_k$ are shorter than $\mathcal{L}(R_k)$. Now the recursion for $R_k$ is equivalent to
\medskip

\hspace{2em} $R_k = 3R_{k-1} + 2 \log \frac {1}{4\varepsilon^2}$

\medskip
and $R_k$ has the closed form
\medskip

\hspace{2em} $R_k = 3^k R_0 +(3^k-1) \log\frac{1}{4\varepsilon^2}.$

\medskip
But the iteration stops at $k = 3g-3+n$ at the latest because this is the maximal possible number of small geodesics for $S$ of genus $g$ with $n$ cusps. Hence we have the following.

\begin{theorem}\label{thm:lq3} Let $S$ be a hyperbolic surface of genus $g$ with $n$ cusps and $B_{\rho}$ a disk of radius $\rho$ in the $\varepsilon$-thick part of $S$, where $0 < \rho < \varepsilon \leq \sfrac{1}{3}$. Then there exists a point $p \in B_{\rho}$ such that
\begin{equation*}
\dist(p,\BS) \geq e^{-3^\kappa R},
\end{equation*}
where $\kappa = 3g-3+n$ and $R = 2 \log \frac{1}{\rho} + M(g-1+\sfrac{n}{2})$ with $M=\frac{195}{\varepsilon^2}\log \frac{134}{\varepsilon}$.
\end{theorem}

\appendix
{\sc Appendix A:} {\it Hyperbolic triangles}
\setcounter{section}{1}
\setcounter{theorem}{0}

For the convenience of the reader we gather a number of properties of hyperbolic triangles inscribed in a circle that are certainly well known but not easily accesible in the literature.
 \begin{lemma}\label{lem:MinimalAngle}%
Consider, for given $\eps > 0$, a hyperbolic triangle with sides of lengths $\geq \eps$ inscribed in a circle of radius $\leq \eps$. Then all angles are bounded from below by $\varphi_{\eps}$, where
 \begin{equation}\label{eq:XXXX}%
\cot{\frac{\varphi_{\eps}}{2}} = \cosh(\eps) \cdot\left \{ \sqrt{1 + 2 \cosh{\eps}} + \sqrt{2 + 2 \cosh{\eps}} \right\}.
 \end{equation}%
\end{lemma}%
For the proof we shall show that the smallest possible angle is $\varphi_{\eps}$ and is achieved in the case as shown on the left hand side in Figure~\ref{Fig:InscribedTriangles}, where we have an isosceles triangle $ABC$ with base $AB$ of length $\eps$ inscribed in a circle of center $O$ and radius $\eps$. An intermediate step in the proof is the following property that does not hold in Euclidean geometry.

 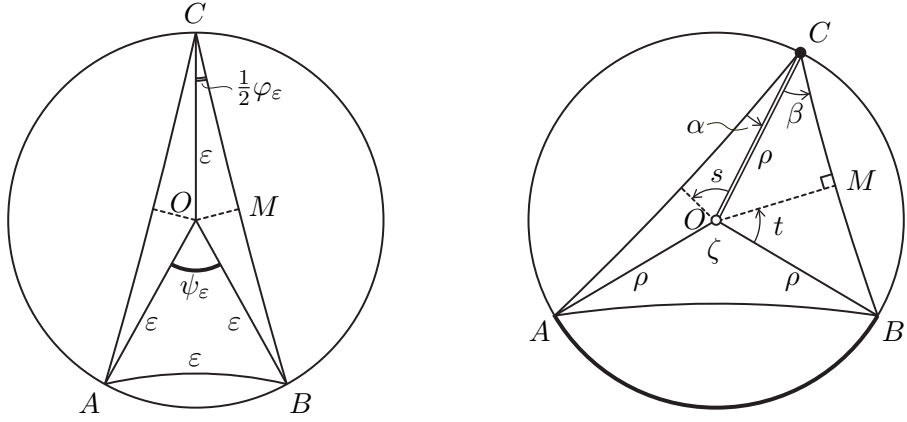
\begin{figure}[ht]
 \vspace{0pt}
 \begin{center}
 \leavevmode
 \SetLabels
 (0.10*-0.01) $A$\\
 (0.33*-0.01) $B$\\
 (0.215*1.02) $C$\\
 (0.2*0.52) $O$\\
 (0.759*0.477) $O$\\
 (0.215*0.115) $\eps$\\
 (0.227*0.65) $\eps$\\
 (0.168*0.21) $\eps$\\
 (0.258*0.21) $\eps$\\
 (0.285*0.83) $\frac12 \varphi_{\eps}$\\
 (0.215*0.305) $\psi_{\eps}$\\
 (0.85*0.46) $t$\\
 (0.78*0.4) $\zeta$\\
 (0.59*0.18) $A$\\
 (0.975*0.18) $B$\\
 (0.895*0.97) $C$\\
 (0.868*0.755) $\beta$\\
 (0.76*0.73) $\alpha$\\
 (0.7*0.33) $\rho$\\
 (0.865*0.33) $\rho$\\
 (0.835*0.65) $\rho$\\
 (0.783*0.60) $s$\\
 (0.29*0.51) $M$\\
 (0.94*0.58) $M$\\
 \endSetLabels
 \AffixLabels{
 \includegraphics{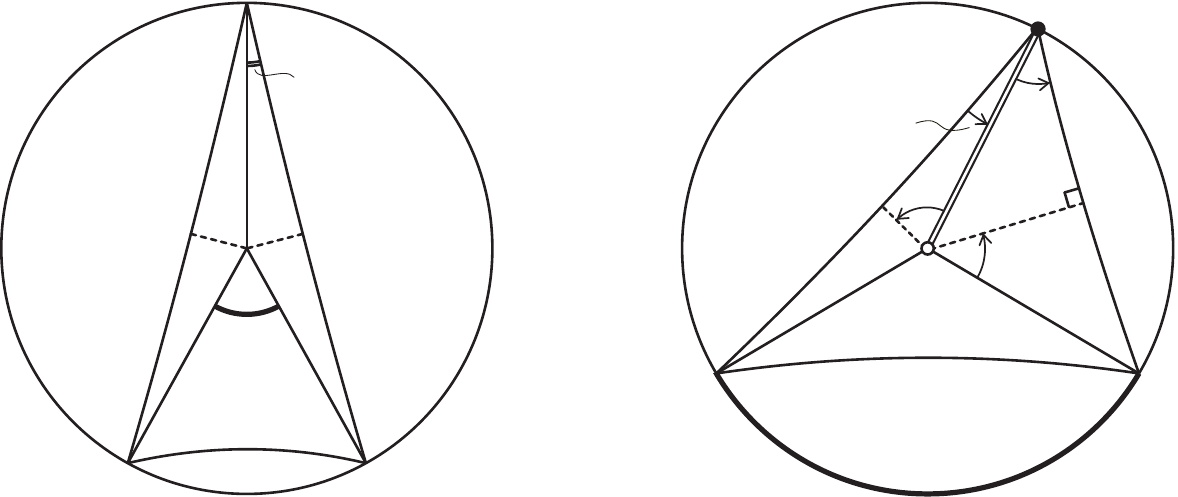} }
 \end{center}
 \vspace{-0pt}
 \caption{\label{Fig:InscribedTriangles} Two different triangles}
 \end{figure}

 \begin{lemma}\label{lem:Isosceles}%
Let $ABC$ be an arbitrary hyperbolic triangle inscribed in some circle of radius $\rho > 0$. If we allow $C$ to move on the circle without crossing $A$ or $B$, then the angle $\gamma$ at $C$ becomes minimal if and only if $\dist(A,C) = \dist(B,C)$.
 \end{lemma}%
 \begin{proof}%
We let $\zeta \in {}]0,2\pi{}[$ be the measure of the angular region between $OA$ and $OB$ into which $C$ is not allowed to enter. As a parameter for the position of $C$ we then take the angle $2t \in {}]0,2\pi - \zeta[{}$ from $OB$ to $OC$. The orthogonal line from $O$ to $BC$ decomposes the triangle $OBC$ into two isometric right-angled triangles $OMC$, $OMB$ with the oriented angle $\beta = \beta(t)$ at $C$, the orientation being from $CO$ to $CB$. 

In a similar way we let $2s \in {}]0,2\pi - \zeta[{}$ be the angle between $OC$ and $OA$. The orthogonal line from $O$ to $AC$ decomposes the triangle $OAC$ into two isometric right-angled triangles with oriented angle $\alpha = \alpha(s)$ at $C$, the orientation being from $CA$ to $CO$. With these orientation conventions we have
 \begin{equation*}
2t + 2s + \zeta = 2\pi
 \end{equation*}%
For the right-angled triangles the following formulas hold:
 \begin{equation*}
\cosh \rho = \cot{\beta} \cot(t); \quad \cosh \rho = \cot{\alpha} \cot(s).
 \end{equation*}%
Indeed, for the configuration as in Figure~\ref{Fig:InscribedTriangles}, where $s,t,\alpha,\beta \in {}]0,\frac{\pi}{2}[$ these are instances of {\bf formula \cite[2.2.2(ii)]{BuserBook}}, and using that $\cot(\pi-s) = - \cot(s)$ we easily see that the formulas remain valid in the cases where $\alpha \leq 0$ or $\beta \leq0$. One also may check that in all cases the angle $\gamma$ of triangle $ABC$ at $C$ satisfies
 \begin{equation*}
\gamma = \alpha + \beta.
 \end{equation*}%
A straightforward computation involving trigonometric identities (including $\tan(s) + \tan(t) = \frac{\sin(s+t)}{\cos(s)\cos(t)}$) now yields the following formula for $\gamma$ as a function of $t \in {}]0,\pi - \frac{\zeta}{2}[$\,,
 \begin{equation*}
\cot \gamma(t) = \cot \tfrac{\zeta}{2} \cdot \cosh \rho -\frac12 \frac{{\sinh}^2\rho}{\sin \tfrac{\zeta}{2}\cosh \rho}%
\left \{ \cos(2t + \tfrac{\zeta}{2} ) + \cos \tfrac{\zeta}{2} \right\}.
 \end{equation*}%
Thus, $\gamma$ becomes minimal when $2t = \pi - \frac{\zeta}{2}$.
\end{proof}%
\begin{proof}[Proof of Lemma~\ref{lem:MinimalAngle}]
By Lemma~\ref{lem:Isosceles} it suffices to look at triangles $ABC$ that are isosceles at $C$. If we move $A$, $B$ on the given circle towards each other until their distance is $\eps$ the angle $\gamma$ decreases. If we then increase the height of the isosceles triangle keeping the base $AB$ fixed until the circumcircle reaches radius $\eps$ the angle $\gamma$ decreases again. Hence, the comparison triangle with the minimal angle $\varphi_{\eps}$ as on the left hand side in Figure~\ref{Fig:InscribedTriangles}.

To compute $\varphi_{\eps}$ we begin with the angle $\psi_{\eps}$ at $O$ of the equilateral triangle $OAB$. By {\bf formula \cite[2.2.2(iii)]{BuserBook}}we have \, $\sinh \frac{\eps}{2} = \sin \frac{\psi_{\eps}}{2} \cdot \sinh{\eps}$ or equivalently, using the half angle formulas for $\cot$ and $\sinh$,
 \begin{equation}\label{eq:PsiEps}%
\cot \frac{\psi_{\eps}}{4}= \sqrt{1+2\cosh \eps} + \sqrt{2 + 2\cosh\eps}.
 \end{equation}%
In the right-angled triangle $OMC$ the angles at $O$ and $C$ are $\frac{1}{2}\pi - \frac{1}{4}\psi_{\eps}$ and $\frac{1}{2}\varphi_{\eps}$, respectively. By the formula used earlier,
 \begin{equation*}%
\cosh \eps = \cot \tfrac{\varphi_{\eps}}{2} \cdot \cot(\tfrac{\pi}{2}-\tfrac{\psi_{\eps}}{4})=
\cot \tfrac{\varphi_{\eps}}{2} \cdot \tan \tfrac{\psi_{\eps}}{4},
 \end{equation*}%
which implies the formula of the lemma.
\end{proof}
For completeness we add an upper bound on the angles of an inscribed triangle. The argument for this is easy: the largest angle is reached for the triangle $ABD$ inscribed in a circle of radius $\eps$ with sides $AB$ and $BD$ of length $\eps$. Triangles $OAB$ and $OBD$ are equilateral with the interior angles $\psi_{\eps}$ and so we have
 \begin{lemma}\label{lem:MaximalAngle}%
Consider, for given $\eps > 0$, a hyperbolic triangle with sides of lengths $\geq \eps$ inscribed in a circle of radius $\leq \eps$. Then all angles are bounded from above by $2\psi_{\eps}$, where $\psi_{\eps}$ is as in~\eqref{eq:PsiEps}.
\end{lemma}%

{\sc Appendix B:} {\it A combinatorial lemma}
\setcounter{section}{2}
\setcounter{theorem}{0}

To distribute up to $L$ identical objects into $K$ distinct boxes is possible in $\binom{L+K}{L}$ different ways; a simple bound is $\frac{(L+K)^K}{K!}$. Here we prove a bound for the case for distributions in packets.

 \begin{lemma}\label{lem:comblem}%
For any string of positive integers $\vec{s} = (s_1, \dots, s_K) \in (\N \setminus \{0\})^K$ and any $L \in \N\setminus \{0\}$ we denote by $\chi_{\vec{s}}(L)$ the number of ordered sequences $\CG{(n_1, \dots n_K)} \in \N^K$ satisfying
 \begin{equation*}%
n_1 s_1 + \dots + n_K s_K \leq L.
 \end{equation*}%
Then
 \begin{equation*}%
\chi_{\vec{s}}{\,} \leq \frac{1}{K! \, s_1 \cdots s_K}(L+s_1+\dots+s_K)^K.
 \end{equation*}%
 \end{lemma}%
 \begin{proof}%
By induction over $K$. For $K=1$ the inequality is clear. For the step from $K$ to $K+1$ we abbreviate $s_{K+1} = s$, $n_{K+1} = n$. The possible values for $n$ are $n = 0,1,2,\dots, \left[\frac{L}{s}\right]$ (largest integer $\leq \frac{L}{s}$). Now we observe that
 \begin{equation*}%
\chi_{(s_1, \dots, s_{K+1})} = \sum_{n=0}^{\left[\frac{L}{s}\right]} \chi_{\vec{s}}(L-n \, s).
 \end{equation*}%
By the induction hypothesis, for any $n = 0, \dots, \left[\frac{L}{s}\right]$ the number of sequences $n_1, \dots, n_K$ has the upper bound 
 \begin{equation*}%
\chi_{\vec{s}}(L- n\, s) \leq \frac{1}{K! \, s_1 \cdots s_K} (L-n\, s + A)^K,
 \end{equation*}%
where we have abbreviated $A= s_1 + \dots + s_K$. \bigskip
It remains to prove that
 \begin{equation*}%
\sum_{n=0}^{\left[\frac{L}{s}\right]} \frac{1}{K! s_1 \cdots s_K}(L-n\, s +A)^K \leq \frac{1}{(K+1)! \, s_1 \cdots s_K s}( L+ A +s)^{K+1}.
 \end{equation*}%
Now
 \begin{align*}%
\sum_{n=0}^{\left[\frac{L}{s}\right]}( L+ A - n\, s)^{K} &= \sum_{n=0}^{\left[\frac{L}{s}\right]} \int_{n-1}^n( L+ A - n\, s)^{K} dt 
\leq \sum_{n=0}^{\left[\frac{L}{s}\right]} \int_{n-1}^n( L+ A - t\, s)^{K} dt \\
&\leq
\int_{-1}^{L/s}(L+A-t \, s)^K dt
= \frac{1}{s(K+1)}\left( (L+A+s)^{K+1}-A^{K+1}\right)\\
&\leq \frac{1}{s(K+1)} (L+A+s)^{K+1}
 \end{align*}%
and the above inequality follows.
 \end{proof}%

There is also a real valued version of the lemma:

 \begin{corollary}\label{cor:comblem}%
For any string of positive real numbers $\vec{\rho} = (\rho_1, \dots, \rho_K) \in \R ^K$ and any $\lambda>0$ we denote by $\chi_{\vec{\rho}}(\lambda)$ the number of ordered sequences $(n_1, \dots n_K) \in \N^K$ satisfying
 \begin{equation*}%
n_1 \rho_1 + \dots + n_K \rho_K \leq \lambda.
 \end{equation*}%
Then
 \begin{equation*}%
\chi_{\vec{\rho}}{\,} \leq \frac{1}{K! \, \rho_1 \cdots \rho_K}(\lambda+\rho_1+\dots+\rho_K)^K.
 \end{equation*}%
 \end{corollary}%

 \begin{proof}%
Take $t>0$ smaller than $\rho_1, \dots, \rho_K$ and set $L = \left[\frac{\lambda}{t}\right]+1$, $s_k = \left[ \frac{\rho_k}{t}\right]$, for $k = 1, \dots, K$. Any string satisfying $n_1 \rho_1 + \dots + n_K \rho_K \leq \lambda$ then also satisfies $n_1 s_1 + \dots + n_K s_K \leq L$. We may thus apply Lemma~\ref{lem:comblem} and the corollary is obtained in the limit as $t \to 0$.
 \end{proof}%

\addcontentsline{toc}{section}{References}

\def\cprime{$'$}

{\em Addresses:}

Institute of Mathematics, EPFL, Switzerland\\
{\em Email:} \href{mailto:peter.buser@epfl.ch}{peter.buser@epfl.ch}\\
Department of Mathematics, University of Luxembourg, Luxembourg\\
{\em Email:} \href{mailto:hugo.parlier@uni.lu}{hugo.parlier@uni.lu}\\

\end{document}